\documentclass[11pt,a4paper]{article}

\usepackage[T1]{fontenc}
\usepackage{lmodern}
\usepackage{microtype}
\usepackage{amsmath,amssymb,amsthm,mathtools}
\usepackage{enumitem}
\usepackage{authblk}
\usepackage[margin=1in]{geometry}
\usepackage{xcolor}
\usepackage[colorlinks=true,linkcolor=blue!55!black,citecolor=blue!55!black,urlcolor=blue!55!black]{hyperref}
\hypersetup{
  pdftitle={A Sharp Spectral Erdos--Ko--Rado Theorem for Uniform Hypergraphs},
  pdfauthor={Mengyu Cao, Mei Lu, Haixiang Zhang},
  pdfsubject={Spectral extremal problems for intersecting uniform families},
  pdfkeywords={Erdos--Ko--Rado theorem, intersecting family, adjacency tensor, spectral radius, complete intersection}
}

\newtheorem{theorem}{Theorem}[section]
\newtheorem{lemma}[theorem]{Lemma}
\newtheorem{proposition}[theorem]{Proposition}
\newtheorem{corollary}[theorem]{Corollary}
\newtheorem{conjecture}[theorem]{Conjecture}

\theoremstyle{remark}

\newcommand{\F}{\mathcal F}
\newcommand{\G}{\mathcal G}
\newcommand{\A}{\mathcal A}
\newcommand{\D}{\mathcal D}
\newcommand{\C}{\mathcal C}
\newcommand{\R}{\mathbb R}
\newcommand{\dd}{\,\mathrm d}
\newcommand{\dotcupbig}{\mathop{\dot\bigcup}}

\title{A Sharp Spectral Erd\H{o}s--Ko--Rado Theorem for Uniform Hypergraphs}
\author[1]{Mengyu Cao\thanks{E-mail: \texttt{myucao@ruc.edu.cn}. Supported by the National Natural Science Foundation of China (12301431) and Beijing Natural Science Foundation (1262010).}}
\author[2]{Mei Lu\thanks{E-mail: \texttt{lumei@tsinghua.edu.cn}. M. Lu is supported by the National Natural Science Foundation of China (Grant 12571372) and Beijing Natural Science Foundation (Grant 1262010).}}
\author[2]{Haixiang Zhang\thanks{Corresponding author. E-mail: \texttt{zhang-hx22@mails.tsinghua.edu.cn}.}}

\affil[1]{\small Institute for Mathematical Sciences, Renmin University of China, Beijing 100086, China}
\affil[2]{\small Department of Mathematical Sciences, Tsinghua University, Beijing 100084, China}
\date{}

\begin{document}

\maketitle

\begin{abstract}
The spectral Erd\H{o}s--Ko--Rado problem asks for the largest
adjacency-tensor spectral radius of a $t$-intersecting $k$-uniform family.
Keevash, Lenz and Mubayi proved that, for fixed $k,t$ and sufficiently large
$n$, the unique extremal family is a full $t$-star, and asked whether such a theorem
extends to all $n$.  Let
$\A_r=\{F\in\binom{[n]}k:|F\cap[t+2r]|\ge t+r\}$ be the Frankl families and
write $\rho_r$ for their spectral radii.  For $2\le t<k$ and $n>2k-t$, we
prove that $\A_0$ is spectrally extremal if and only if $\rho_0\ge\rho_1$; it
is unique up to permutation when the inequality is strict, whereas $\A_0$
and $\A_1$ are both extremal at equality.  The layerwise pull used in the
Ahlswede--Khachatrian cardinality proof is not applicable here: applied
directly, it may decrease the spectral radius.  Our proof instead pulls all
boundary layers simultaneously and applies Perron tail symmetrization.  It
follows that $\A_0$ is uniquely extremal for
$n\ge (t+1)(k-t+1)+\lceil(t+1)\log(t+1)\rceil$; the leading coefficient $t+1$
is best possible for fixed $t$.  We also determine all extremal structures
for $t=1$ throughout the range $n\ge2k$.
\end{abstract}



\section{Introduction}

For a finite set $X$, write $\binom Xk$ for the collection of its $k$-element
subsets and let $[n]=\{1,\ldots,n\}$.  A family
$\F\subseteq\binom{[n]}k$ is \emph{$t$-intersecting} if
$|F\cap F'|\ge t$ for every $F,F'\in\F$.  Throughout, ``up to permutation''
means up to applying a permutation of the ground set.  The
Erd\H{o}s--Ko--Rado theorem states that, for $n\ge2k$, an intersecting family
has at most $\binom{n-1}{k-1}$ members; for $n>2k$, equality forces a full
star~\cite{EKR}.  Wilson obtained the corresponding full $t$-star theorem for
large ground sets~\cite{Wilson}.  Ahlswede and Khachatrian then settled every
parameter range in their complete intersection theorem~\cite{AK}: the maximum
cardinality is attained by one of
\begin{equation}\label{eq:AK-candidates}
 \A_r=\A_{t,r}(n,k)
 :=\left\{F\in\binom{[n]}k:|F\cap[t+2r]|\ge t+r\right\},
 \qquad 0\le r\le k-t.
\end{equation}
These are the \emph{Frankl families}.  The two endpoints are the full
$t$-star $\A_0$ and the complete $k$-graph on $2k-t$ vertices
$\A_{k-t}$; the intermediate families describe the transitions between
them.  Generating families and pushing--pulling arguments are central to the
cardinality theory; see~\cite{AKPushPull,FranklTokushige}.

Tensor eigenvalues, introduced independently by Qi and Lim~\cite{Qi,Lim},
make it possible to replace cardinality by a spectral objective.  Following
Cooper and Dutle~\cite{CooperDutle}, for a $k$-uniform family $\F$ and
$x=(x_1,\ldots,x_n)\in\R_{\ge0}^n$ let
\begin{equation}\label{eq:spectral-polynomial}
 P_{\F}(x):=k\sum_{F\in\F}\prod_{i\in F}x_i,
 \qquad
 \rho(\F):=\max_{\substack{x\ge0\\\sum_i x_i^k=1}}P_{\F}(x).
\end{equation}
Thus $\rho(\F)$ is the adjacency-tensor spectral radius in the standard
normalization.  A maximizing vector will be called a \emph{Perron vector}; its
existence follows from compactness, and the terminology reflects the
Perron--Frobenius theory for nonnegative tensors and multilinear
forms~\cite{FriedlandGaubertHan}.  A $t$-intersecting family is
\emph{spectrally extremal} if its spectral radius is largest among all
$t$-intersecting subfamilies of $\binom{[n]}k$.  Unlike cardinality,
$\rho(\F)$ also optimizes the distribution of weight among the vertices, so
ordinary edge comparisons do not by themselves control the spectral order.

Keevash, Lenz and Mubayi developed a general transfer framework for spectral
extremal problems and proved that, for fixed $k,t$ and all sufficiently large
$n$, the full $t$-star is, up to permutation, the unique spectral extremal
structure among $t$-intersecting families~\cite{KLM}.  Zhang and Zhang later
considered the adjacency $\mathcal A_\alpha$-tensor and the incidence
$\mathcal Q$-tensor~\cite{ZhangZhang}.  The sharp finite-parameter problem
remained open: one must determine the spectral order of the Frankl families
and, more importantly, prove that no other $t$-intersecting family has larger
spectral radius.

We solve this problem in the full-star regime.  Write
\[
 \rho_r:=\rho(\A_r),\quad 0\le r\le k-t.
\]
The two endpoint radii are given by
\begin{equation}\label{eq:endpoint-radii}
 \rho_0=\binom{n-t}{k-t}\left(\frac {k-t}{n-t}\right)^{\frac {k-t}{k}},
 \qquad
 \rho_{k-t}=\binom{2k-t-1}{k-t},
\end{equation}
where the first expression is interpreted as $1$ when $k-t=0$.  The
two-orbit symmetry of $\A_r$ reduces every $\rho_r$ to a one-variable
maximum; see Proposition~\ref{prop:candidate-orbit-formula}.  Our first theorem
shows that, for $t\ge2$, the comparison of only the first two radii decides the
global problem.

\begin{theorem}\label{thm:first-candidate}
Let $2\le t<k$ and $n>2k-t$.  Then the full $t$-star is spectrally extremal
among the $t$-intersecting subfamilies of $\binom{[n]}k$ if and only if
\begin{equation}\label{eq:first-candidate-criterion}
 \rho_0\ge\rho_1.
\end{equation}
If $\rho_0>\rho_1$, then $\A_0$ is the unique  extremal structure among the $t$-intersecting
subfamilies of $\binom{[n]}k$, up to permutation.  If $\rho_0=\rho_1$, then both $\A_0$ and $\A_1$
are spectral extremal structures, so the full $t$-star is not unique.
\end{theorem}

Thus the unspecified large-$n$ condition is replaced by a sharp, directly
computable criterion: $\A_1$ is a strict counterexample when
$\rho_0<\rho_1$, while equality gives two inequivalent extremal structures.
The criterion also yields an explicit threshold involving only $n,k,t$.
Here $\log$ denotes the natural logarithm.

\begin{corollary}\label{cor:explicit-threshold}
Let $2\le t<k$.  If
\begin{equation}\label{eq:explicit-threshold}
 n\ge (t+1)(k-t+1)
      +\left\lceil (t+1)\log(t+1)\right\rceil,
\end{equation}
then the full $t$-star is spectrally extremal among the $t$-intersecting
subfamilies of $\binom{[n]}k$, and it is the unique extremal structure up to permutation.
\end{corollary}

For fixed $t$, the threshold in~\eqref{eq:explicit-threshold} is
$(t+1)(k-t)+O_t(1)$.  The coefficient $t+1$ is best possible at leading
order: the necessary estimate~\eqref{eq:basic-numerical-condition} shows that
if $n=c(k-t)+O_t(1)$ and the full $t$-star is spectrally extremal, then
$c\ge t+1$.

The case $t=1$ requires separate treatment.  Zhang and Zhang stated a
spectral EKR upper bound for the $\mathcal A_\alpha$-tensor of intersecting
$k$-uniform hypergraphs in the range $n\ge2k$~\cite{ZhangZhang}.  At
$\alpha=0$, their statement specializes to an adjacency-tensor bound attained
by the full star, but this specialization fails at the endpoint $n=2k$ under
the normalization~\eqref{eq:spectral-polynomial}.  Indeed, the last Frankl
family
$
\A_{k-1}=\binom{[2k-1]}k
$
satisfies
$
\left(\frac{\rho_{k-1}}{\rho_0}\right)^k
=\frac{k^k}{(2k-1)(k-1)^{k-1}}>1
$
for every $k\ge2$.  Thus the endpoint failure is systematic rather than a
sporadic example.  The correct statement combines our boundary pull for
$n\ge2k+1$ with the sharp edge-count spectral theorem of Bai and Lu and, for
$k=2$, Stanley's graph theorem~\cite{BaiLu,Stanley}.

\begin{theorem}\label{thm:t-one}
Let $k\ge2$ and $n\ge2k$.
\begin{enumerate}[label=\textup{(\arabic*)},leftmargin=2.2em]
 \item If $n\ge2k+1$, then every intersecting
 $\F\subseteq\binom{[n]}k$ satisfies $\rho(\F)\le\rho_0$.

 If
 $(n,k)\ne(5,2)$, then $\rho(\F)=\rho_0$  if and only if $\F$ is a full star, up to
 permutation.

  If $(n,k)=(5,2)$, then $\rho(\F)=\rho_0$ if and only if, up to
 permutation, $\F$ is the full star $K_{1,4}$ or the triangle $K_3$ together
 with two isolated vertices.
 \item If $n=2k$, then every intersecting $\F\subseteq\binom{[2k]}k$ satisfies
 \[
  \rho(\F)\le \binom{2k-2}{k-1}=\rho_{k-1},
 \]
 with equality if and only if, up to permutation,
 $\F=\binom{[2k-1]}k$.
\end{enumerate}
\end{theorem}

The main methodological point is that our pull is not the
Ahlswede--Khachatrian pull transplanted to a spectral objective.  Their
cardinality argument can compare boundary layers separately.  In the spectral
problem, however, deleting the last support coordinate changes the Perron
weights and the contributions of all trace classes simultaneously; a pull
that is profitable on one layer need not compensate for the losses on the
others.  Consequently, the Ahlswede--Khachatrian layerwise pull does not
directly provide a monotone operation for spectral radius.

We overcome this obstruction with a \emph{simultaneous boundary pull}.
Starting from a minimal generating antichain supported on $[s]$, an exact
boundary-trace decomposition records every contribution involving the last
support coordinate.  The boundary ranks are organized into complementary
orbits: from each off-diagonal pair we select a profitable rank, while the
diagonal rank requires a separate common-omission selection.  All selected
shadows are then pulled at once, which both preserves $t$-intersection and
compensates for every discarded boundary layer.  At the same time, Perron
tail symmetrization redistributes the $\ell_k$-mass of
$x_s,x_{s+1},\ldots,x_n$ uniformly over the new tail.  The resulting global
polynomial ledger proves that the generated family on $[s-1]$ has no smaller
spectral radius.  Iterating this support reduction converts the comparison
among the Frankl families into a bound for every $t$-intersecting family.  Its
strict form also yields the uniqueness assertion in
Theorem~\ref{thm:first-candidate}.

The paper is organized as follows.  Section~\ref{sec:generators} introduces the
generating-family reduction, exact trace decomposition and Perron
symmetrization.  Section~\ref{sec:numerics} derives the one-variable formula
for the Frankl-family radii and proves the required monotonicity.
Section~\ref{sec:proof} constructs the simultaneous pull and deduces
Theorem~\ref{thm:first-candidate} and
Corollary~\ref{cor:explicit-threshold}.  Section~\ref{sec:t-one} treats
$t=1$, and Section~\ref{sec:conjecture} states the spectral complete
intersection conjecture.

\section{Generating families and Perron symmetrization}\label{sec:generators}

\subsection{Generating antichains and boundary pairs}

For $T\subseteq[n]$ with $|T|\le k$, let
\[
 \D(T):=\left\{F\in\binom{[n]}k:T\subseteq F\right\}.
\]
We call $T$ a \emph{generator} of $\F$ if $\D(T)\subseteq\F$. Let $\G$ be the set of inclusion-minimal generators of $\F$. Then $\G$ forms an antichain, and we call $\G$  the
\emph{minimal generating antichain}. Then
\[
 \F=\bigcup_{G\in\G}\D(G).
\]
A set $T\subseteq[n]$ satisfying $|T\cap F|\ge t$ for every $F\in\F$ is called a
\emph{$t$-transversal} of $\F$.

We first collect the standard generator facts.  Their usual proofs are part of
the classical generating-family method; the short extension observation below
explains why the hypothesis $n>2k-t$ is the correct one.

\begin{lemma}\label{lem:standard-generator}
Let $A\subseteq[n]$, $|A|\le k$, and $F\in\binom{[n]}k$.  Then
\begin{equation}\label{eq:extension}
 \min_{H\in\D(A)}|H\cap F|=\max\{|A\cap F|,\,2k-n\}.
\end{equation}
Consequently, when $n>2k-t$ the following statements hold.
\begin{enumerate}[label=\textup{(\arabic*)},leftmargin=2.2em]
 \item Every generator of a $t$-intersecting family is a $t$-transversal.
 \item If $\F$ is inclusion-maximal subject to being $t$-intersecting, then
 a set is a generator if and only if it is a $t$-transversal.
 \item The minimal generating antichain of a $t$-intersecting family is itself
 $t$-intersecting.
\end{enumerate}
\end{lemma}

\begin{proof}
To obtain~\eqref{eq:extension}, extend $A$ first with points outside $A\cup F$;
points of $F\setminus A$ are forced only after all such points have been used.
If $|A\cap F|<t$, the assumption $n>2k-t$ therefore supplies an extension
$H\supseteq A$ with $|H\cap F|<t$.  This proves (1).  Conversely, if $T$ is a
$t$-transversal, then every $k$-set containing $T$ can be added while
preserving $t$-intersection, which proves (2) by maximality.  Finally, apply
(1) to one generator and the extension observation to the other to obtain
(3).  These are the standard generating-family arguments; see, for example,
\cite{AK,FranklTokushige}.
\end{proof}

For a generating antichain $\G$, define its \emph{support length} by
\[
 \sigma(\G):=\max\{i:i\in\cup_{G\in\G}G\}.
\]
Thus every generator is contained in $[\sigma(\G)]$, and $\sigma(\G)$ is the
largest coordinate needed to support the generating antichain.

We will use two elementary rigidity facts when passing to a compressed
extremal structure.  The first is the strict Perron--Frobenius monotonicity of a full
star; the second says that, one step beyond the boundary $n=2k-t+1$, shifting
cannot turn a genuinely different family into a full star.

We will repeatedly use the following Perron--Frobenius comparison theorem.
If $0\le A\le B$ entrywise for two nonnegative tensors, then
$\rho(A)\le\rho(B)$; if, in addition, $B$ is weakly irreducible and $A\ne B$,
then $\rho(A)<\rho(B)$~\cite[Theorem~3.4]{KannanShakedBerman}.  For an
adjacency tensor, weak irreducibility is equivalent to weak connectivity of
the corresponding uniform hypergraph.

\begin{lemma}\label{lem:star-strict-monotonicity}
Let $1\le t<k$,  $T\in\binom{[n]}t$ and 
$\mathcal B\subsetneq\D(T)$.  Then
\[
 \rho(\mathcal B)<\rho(\D(T)).
\]
\end{lemma}

\begin{proof}
The hypergraph $\D(T)$ is weakly connected: every vertex outside $T$ occurs
in an edge together with all vertices of $T$.  Its adjacency tensor is
therefore weakly irreducible.  Since the adjacency tensor of $\mathcal B$ is
entrywise at most and not equal to that of $\D(T)$, the strict part of the
comparison theorem above gives the claim.
\end{proof}

\begin{lemma}\label{lem:shift-rigidity}
Let $1\le t<k$ and $n\ge2k-t+2$.  If an ordinary $(i,j)$-shift of a
$t$-intersecting family $\mathcal B\subseteq\binom{[n]}k$ is a full
$t$-star, then $\mathcal B$ itself is a full $t$-star, possibly with a
different center.
\end{lemma}

\begin{proof}
Write the shifted family as $\D(T)$, where $|T|=t$.  The output of an
$(i,j)$-shift is $(i,j)$-shifted, so the case $j\in T$ and $i\notin T$ is
impossible.  If $i,j$ are both in $T$ or both outside $T$, the membership of
each two-member shift pair is forced by $\D(T)$, and hence
$\mathcal B=\D(T)$.

It remains to consider $i\in T$ and $j\notin T$.  Let
$
 R:=T\setminus\{i\}.
$
For each $X\in\binom {[n]\setminus(R\cup\{i,j\})} {k-t}$, exactly one of
\[
 R\cup X\cup\{i\},\qquad R\cup X\cup\{j\}
\]
belongs to $\mathcal B$; these are precisely the two possible preimages of
the corresponding shifted edge.  Color each $X\in\binom {[n]\setminus(R\cup\{i,j\})} {k-t}$ by the chosen endpoint $i$ or
$j$.  If disjoint $(k-t)$-sets received opposite colors, the corresponding two
members of $\mathcal B$ would intersect in exactly $R$, of size $t-1$.
Thus no edge of the Kneser graph on $\binom {[n]\setminus(R\cup\{i,j\})} {k-t}$ joins opposite colors.

Now $|[n]\setminus(R\cup\{i,j\})|=n-t-1\ge2(k-t)+1$, and this Kneser graph is connected.  Indeed, two
$(k-t)$-sets differing in one element have a common disjoint $(k-t)$-set, and any
two $(k-t)$-sets are joined by a sequence of single-element exchanges.  Hence
all $(k-t)$-sets have the same color.  The family $\mathcal B$ is consequently
either $\D(T)$ or the full star with center $R\cup\{j\}$.
\end{proof}

The next reduction selects an extremal pair for which maximality, compression,
support minimality, and the ordering of the Perron coordinates are compatible.

\begin{proposition}\label{prop:perron-shift}
Suppose that $n\ge2k-t+2$ and there exists a spectral extremal structure that
is not a full $t$-star among all
$t$-intersecting $k$-uniform families in $[n]$. Then one may choose a non-star spectral extremal
structure $\F$ and a Perron vector $x=(x_1,\ldots,x_n)^T$ of $\F$ such that:
\begin{enumerate}[label=\textup{(\arabic*)},leftmargin=2.2em]
 \item $\F$ is inclusion-maximal and left-compressed;
 \item if $\G$ is the minimal generating antichain of $\F$, then
 $\sigma(\G)$ is smallest among all non-star spectral extremal structures
 satisfying \textup{(1)};
 \item $x_1\ge\cdots\ge x_n\ge0$;
 \item every vertex contained in an edge of $\F$ has a positive coordinate
 in $x$.
\end{enumerate}
\end{proposition}

\begin{proof}
Let $\rho^*$ be the largest spectral radius among all
$t$-intersecting $k$-uniform families.
We first show that there is an inclusion-maximal, left-compressed non-star
spectral extremal structure. Choose a non-star spectral extremal structure
with the maximum number of edges, and let $x$ be a Perron vector. Apply a
permutation simultaneously to the family and to $x$ so that
\[
 x_1\ge\cdots\ge x_n\ge0.
\]
For $i<j$, the ordinary $(i,j)$-shift preserves $t$-intersection and does not
decrease the spectral objective at $x$, since every monomial that is genuinely replaced has a factor $x_j$
  replaced by $x_i$, where $x_i\ge x_j$. Hence every shifted family
is again spectral extremal. By Lemma~\ref{lem:shift-rigidity}, it remains
non-star. Repeated nontrivial shifts strictly decrease
\[
 \sum_{F\in\F}\sum_{i\in F}i,
\]
so the process terminates at a left-compressed non-star spectral extremal
structure with the same number of edges.

This family is inclusion-maximal. Indeed, if an edge could be added while
preserving $t$-intersection, the resulting family would have spectral radius
at least $\rho^*$ and hence would again be spectral extremal. It cannot be a
full $t$-star, since the original family would then be a proper subfamily of
that star with the same spectral radius, contradicting
Lemma~\ref{lem:star-strict-monotonicity}. This contradicts the maximal choice
of the number of edges.

Now we choose, among all inclusion-maximal, left-compressed non-star spectral
extremal structures, a $t$-intersecting $k$-uniform family $\F$ whose minimal generating antichain $\G$
has the smallest support length. Set
\[
 s:=\sigma(\G).
\]
This gives \textup{(1)} and \textup{(2)}. Choose a Perron vector $x$ of
$\F$. Since $\F$ is left-compressed, the standard coordinate-exchange
argument allows $x$ to be chosen so that
\[
 x_1\ge\cdots\ge x_n\ge0,
\]
proving \textup{(3)}.

Finally, every vertex contained in an edge of $\F$ has a positive coordinate
in $x$. Indeed, the nonisolated part of $\F$ is weakly connected, since any
two edges of a $t$-intersecting family intersect. The Perron--Frobenius
theorem therefore gives a Perron vector that is strictly positive on every
nonisolated vertex. Extending it by zero on the isolated vertices and then
applying the preceding coordinate exchanges preserves this positivity.
Hence
\[
 i\in F\ \text{for some }F\in\F
 \qquad\Longrightarrow\qquad
 x_i>0,
\]
which proves \textup{(4)}.
\end{proof}

Fix a pair $(\F,x)$ as in Proposition~\ref{prop:perron-shift}. Write
$s=\sigma(\G)$ and set
\[
 \G^{\partial}:=\{G\in\G:s\in G\},
 \qquad
 \G_i^{\partial}:=\{G\in\G^{\partial}:|G|=i\},
\]where $\G$ is the minimal generating antichain of $\F$.
The superscript $\partial$ indicates the boundary at the last support
coordinate.

The last support coordinate forces every boundary generator to have a tight
partner, and the two ranks are complementary with respect to $s+t$.

\begin{lemma}\label{lem:tight-pairs}
For every $A\in\G^{\partial}$ there is $B\in\G^{\partial}$ such that
$|A\cap B|=t$.  Moreover, whenever $A,B\in\G^{\partial}$ satisfy
$|A\cap B|=t$, one has
\begin{equation}\label{eq:tight-union}
 A\cup B=[s],
 \qquad |A|+|B|=s+t.
\end{equation}
Consequently,
\begin{equation}\label{eq:rank-pairing}
 \G_i^{\partial}\ne\emptyset
 \quad\Longrightarrow\quad
 \G_{s+t-i}^{\partial}\ne\emptyset 
 \qquad \mbox{and}\qquad s\le2k-t.
\end{equation}
\end{lemma}

\begin{proof}
Let $A\in\G^{\partial}$ and $A^-=A\setminus\{s\}$.  If $A^-$ is a $t$-transversal of $\F$, then it would be
a generator by Lemma~\ref{lem:standard-generator} and would contain a minimal
generator properly contained in $A$, contradicting the antichain property.
Choose $F\in\F$ with $|A^-\cap F|<t$.  Since $A$ is a $t$-transversal, we must
have $|A\cap F|=t$ and $s\in F$.  A minimal generator $B\subseteq F$ then
satisfies $|A\cap B|=t$ and necessarily contains $s$.

Now suppose $A,B\in\G^{\partial}$ and $|A\cap B|=t$.  If
$u\in[s]\setminus(A\cup B)$, let $B'=(B\setminus\{s\})\cup\{u\}$.  Every
$k$-set containing $B'$ belongs to $\F$: when it contains $s$ it also contains
$B$, and otherwise it is obtained by left-compressing a member containing
$B$.  Thus $B'$ is a generator and contains some $C\in\G$.  But
$|A\cap C|\le|A\cap B'|=t-1$, contradicting
Lemma~\ref{lem:standard-generator}(3).  Hence $A\cup B=[s]$, and the remaining
claims follow from the inclusion--exclusion principle.
\end{proof}

\subsection{Exact boundary traces and Perron tail symmetrization}

For a family $\mathcal B\subseteq\binom{[n]}k$ and a set $S\subseteq[n]$,
define the \emph{trace of $\mathcal B$ on $S$} by
\[
 \operatorname{Tr}_S(\mathcal B)
 :=\{B\cap S:B\in\mathcal B\}.
\]
For $A\subseteq S$, the \emph{trace class of $\mathcal B$ on $S$
corresponding to $A$} is
\begin{equation}\label{eq:trace-class-definition}
 \mathcal B_S(A):=\{B\in\mathcal B:B\cap S=A\}.
\end{equation}
Thus $\mathcal B_S(A)\ne\emptyset$ exactly when
$A\in\operatorname{Tr}_S(\mathcal B)$, and
\[
 \mathcal B
 =\dotcupbig_{A\in\operatorname{Tr}_S(\mathcal B)}\mathcal B_S(A).
\]
We also write
\[
 \C_S(A):=\left\{F\in\binom{[n]}k:F\cap S=A\right\}
 =\left\{A\cup Z:Z\in\binom{[n]\setminus S}{k-|A|}\right\},
\]
so that $\mathcal B_S(A)=\mathcal B\cap\C_S(A)$.  When $S=[m]$, we
abbreviate $\operatorname{Tr}_{[m]}$, $\mathcal B_{[m]}(A)$, and
$\C_{[m]}(A)$ to $\operatorname{Tr}_m$, $\mathcal B_m(A)$, and
$\C_m(A)$, respectively.

Retain the pair $(\F,x)$ fixed above, and let $\G$ be the minimal generating
antichain of $\F$.  Write
\[
 \G^{\circ}:=\{G\in\G:s\notin G\},
 \qquad
 \F^{\circ}:=\bigcup_{G\in\G^{\circ}}\D(G).
\]
For any subcollection $\mathcal S\subseteq\G^{\partial}$, define the
generating family and the $k$-uniform family generated by it by
\begin{equation}\label{eq:macro-generators}
 \mathcal H(\mathcal S)
 :=\G^{\circ}\cup
   \{A\setminus\{s\}:A\in\mathcal S\},
 \qquad
 \mathcal Q(\mathcal S)
 :=\bigcup_{H\in\mathcal H(\mathcal S)}\D(H).
\end{equation}
We call the pair
$\bigl(\mathcal H(\mathcal S),\mathcal Q(\mathcal S)\bigr)$ the
\emph{pull at $s$ determined by $\mathcal S$}.  This is a
support-reducing operation: it keeps the generators already supported on
$[s-1]$, replaces each selected boundary generator by the set obtained by
deleting $s$, and discards the unselected boundary generators.  Thus every
member of $\mathcal H(\mathcal S)$ lies in $[s-1]$, so the operation attempts
to lower the last support coordinate from $s$ to at most $s-1$.  The problem
is to choose $\mathcal S$ so that $\mathcal H(\mathcal S)$ remains
$t$-intersecting and $\rho(\mathcal Q(\mathcal S))\ge\rho(\F)$.
Section~\ref{sec:proof} makes this choice simultaneously over all boundary
ranks.

For $A\in\G^{\partial}$, the fact that $A$ is a generator gives
\begin{equation}\label{eq:boundary-trace-class}
 \C_s(A)
 =\left\{A\cup Z:Z\in\binom{\{s+1,\ldots,n\}}{k-|A|}\right\}.
\end{equation}
We refer to the family in~\eqref{eq:boundary-trace-class} as the boundary
trace class corresponding to $A$.

For every $A\in\G^{\partial}$, the family $\C_s(A)$ is nonempty.  Indeed, Lemma
\ref{lem:tight-pairs} gives a partner $B\in\G^{\partial}$ with
$|A|+|B|=s+t$ and $|B|\le k$.  Hence, if we set $|A|=k-q$, then
\begin{equation}\label{eq:q-range}
 0\le q\le2k-s-t<n-s.
\end{equation}

\begin{lemma}\label{lem:trace-decomposition}
Let $\F$ be a $t$-intersecting $k$-uniform family.  Then
\begin{equation}\label{eq:trace-decomposition}
 \F=\F^{\circ}\,\dot\cup\,
 \dotcupbig_{A\in\G^{\partial}}\C_s(A).
\end{equation}
\end{lemma}

\begin{proof}
Equation~\eqref{eq:boundary-trace-class} shows that the displayed classes lie
in $\F$.  Distinct sets $A\in\G^{\partial}$ give distinct trace classes by
\eqref{eq:trace-class-definition}.  These classes are disjoint from
$\F^{\circ}$ by the antichain property.

Take $F\in\F\setminus\F^{\circ}$ and choose
$A\in\G^{\partial}$ with $A\subseteq F$.  Suppose $F\cap[s]\not=A$, say
$u\in(F\cap[s])\setminus A$. Set $T=(A\setminus\{s\})\cup\{u\}$.  We claim
that $T$ is a $t$-transversal.  It is enough to test minimal generators
$B\in\G$.  If $|A\cap B|\ge t+1$, removing $s$ loses at most one point.  If
$|A\cap B|=t$ and $s\notin B$, nothing is lost.  Finally, if
$|A\cap B|=t$ and $s\in B$, Lemma~\ref{lem:tight-pairs} gives
$A\cup B=[s]$, so $u\in B$ replaces the deleted coordinate.  Thus
$|T\cap B|\ge t$ in every case.

By maximality and Lemma~\ref{lem:standard-generator}, $T$ is a generator and
contains a minimal generator $G\in\G^{\circ}$.  Since $T\subseteq F$, this
would imply $F\in\F^{\circ}$, a contradiction.  Hence $F\cap[s]=A$ and
$F\in\C_s(A)$.
\end{proof}

By symmetry of $\F$ under permutations of the last $n-s$ coordinates, its
Perron vector can be chosen so that
\[
 x_{s+1}=\cdots=x_n=\beta,
 \qquad x_s=\gamma.
\]
Indeed, start from any Perron vector and define $\beta$ by
\[
 (n-s)\beta^k=\sum_{a=s+1}^n x_a^k.
\]
For $0\le d\le n-s$, Maclaurin's inequality followed by the power-mean
inequality gives
\[
 \sum_{Z\in\binom{\{s+1,\ldots,n\}}{d}}\prod_{i\in Z}x_i
 \le \binom{n-s}{d}
      \left(\frac{\sum_{a=s+1}^n x_a^k}{n-s}\right)^{d/k}
 =\binom{n-s}{d}\beta^d.
\]
Decomposing $\F$ into its trace classes on $[s]$ therefore shows that
replacing $x_{s+1},\ldots,x_n$ by $\beta$ does not decrease $P_\F(x)$ and
preserves $\sum_i x_i^k$.  The resulting vector is again a Perron vector; we
continue to denote it by $x$ and put $\gamma=x_s$.

To compare $\F$ with a family obtained by the pull, define
$y=(y_1,\ldots,y_n)$ by
\begin{equation}\label{eq:tail-equalization}
 y_i=x_i\quad (i\in[s-1]),
 \qquad
 y_s=y_{s+1}=\cdots=y_n=\alpha,
 \qquad \mbox{where} \qquad(n-s+1)\alpha^k=\gamma^k+(n-s)\beta^k.
\end{equation}
We call this operation \emph{Perron tail symmetrization}.

\begin{lemma}\label{lem:tail-symmetrization}
Let $x$ be the Perron vector above, with
$x_s=\gamma$ and $x_{s+1}=\cdots=x_n=\beta$, and let $y$ be defined by
\eqref{eq:tail-equalization}.  Then
\begin{equation}\label{eq:tail-symmetrization-inequality}
 P_{\F^{\circ}}(x)\le P_{\F^{\circ}}(y).
\end{equation}
\end{lemma}

\begin{proof}
Put $m=n-s+1$.  For every
$A\in\operatorname{Tr}_{s-1}(\F^{\circ})$, the definition of
$\F^{\circ}$ gives
\[
 \F^{\circ}_{s-1}(A)=\C_{s-1}(A).
\]
Consequently,
\begin{align*}
 P_{\F^{\circ}}(x)
 &=k\sum_{A\in\operatorname{Tr}_{s-1}(\F^{\circ})}
   \left(\prod_{i\in A}x_i\right)
   \sum_{Z\in\binom{\{s,\ldots,n\}}{k-|A|}}
   \prod_{j\in Z}x_j,\\
 P_{\F^{\circ}}(y)
 &=k\sum_{A\in\operatorname{Tr}_{s-1}(\F^{\circ})}
   \left(\prod_{i\in A}x_i\right)
   \binom{m}{k-|A|}\alpha^{k-|A|}.
\end{align*}
For $0\le d\le m$, Maclaurin's inequality and the power-mean inequality give
\[
 \sum_{Z\in\binom{\{s,\ldots,n\}}{d}}\prod_{j\in Z}x_j
 \le \binom md
      \left(\frac{\sum_{j=s}^n x_j^k}{m}\right)^{d/k}
 =\binom md\alpha^d.
\]
Termwise comparison in the preceding two sums proves
\eqref{eq:tail-symmetrization-inequality}.
\end{proof}

Let $A\in\G^{\partial}$ with $|A|=k-q$, and put
$A^-=A\setminus\{s\}$.  Then
\begin{equation}\label{eq:old-boundary-contribution}
 P_{\C_s(A)}(x)
 =k\left(\prod_{i\in A}x_i\right)\binom {n-s}{q}\beta^q.
\end{equation}
Assume temporarily that $\gamma,\beta>0$, and set
\[
 \eta:=\frac\beta\gamma,
 \qquad
 \theta:=\frac\alpha\gamma
 =\left(\frac{1+(n-s)\eta^k}{n-s+1}\right)^{1/k}.
\]
\begin{equation}\label{eq:gain-factor}
 \begin{aligned}
 P_{\C_{s-1}(A^-)}(y)
 &=k\left(\prod_{i\in A^-}x_i\right)
   \binom{n-s+1}{q+1}\alpha^{q+1}\\
 &=\Gamma_q(\eta)P_{\C_s(A)}(x),
 \qquad
 \Gamma_q(\eta):=\frac{n-s+1}{q+1}
                  \frac{\theta^{q+1}}{\eta^q}.
 \end{aligned}
\end{equation}
Thus $\Gamma_q(\eta)$ is the ratio between the two explicitly defined
polynomial contributions in~\eqref{eq:gain-factor}.  Its minimum is
independent of the Perron vector.

More generally, for $\mathcal S\subseteq\G^{\partial}$ define
\[
 \mathcal B(\mathcal S)
 :=\dotcupbig_{A\in\mathcal S}\C_s(A),
 \qquad
 \mathcal N(\mathcal S)
 :=\dotcupbig_{A\in\mathcal S}\C_{s-1}(A\setminus\{s\}).
\]
If every $A\in\mathcal S$ has size $k-q$, then linearity of $P$ and
\eqref{eq:gain-factor} give
\begin{equation}\label{eq:old-new-trace-polynomials}
 P_{\mathcal N(\mathcal S)}(y)
 =\Gamma_q(\eta)P_{\mathcal B(\mathcal S)}(x).
\end{equation}

The following calculation identifies this worst-case gain exactly.

\begin{lemma}\label{lem:universal-gain}
For $q\ge1$,
\begin{equation}\label{eq:universal-gain}
 g_q:=\min_{\eta>0}\Gamma_q(\eta)
 =\left(\frac{n-s+1}{q+1}\right)^{(k-q-1)/k}
  \left(\frac {n-s}q\right)^{q/k},
\end{equation}
and the minimum occurs at $(n-s)\eta^k=q$.  For $q=0$,
\begin{equation}\label{eq:gain-q-zero}
 g_0:=\inf_{\eta>0}\Gamma_0(\eta)=(n-s+1)^{(k-1)/k}.
\end{equation}
\end{lemma}

\begin{proof}
For $q\ge1$, differentiation gives
\[
 \frac{\dd}{\dd\eta}\log\Gamma_q(\eta)
 =\frac{(n-s)\eta^k-q}{\eta(1+(n-s)\eta^k)}.
\]
The unique minimum is therefore $(n-s)\eta^k=q$, and substitution yields
\eqref{eq:universal-gain}.  Formula~\eqref{eq:gain-q-zero} follows directly
from~\eqref{eq:gain-factor}.
\end{proof}

If $\gamma=0$, then $\beta=0$ and $P_{\C_s(A)}(x)=0$ for every
$A\in\G^{\partial}$.  If $\gamma>0$ but $\beta=0$, then
$P_{\C_s(A)}(x)=0$ whenever
$q>0$, while the case $q=0$ is covered by~\eqref{eq:gain-q-zero}.  Hence all
subsequent inequalities between the corresponding values of $P$ extend to
these cases by continuity.

\section{Monotonicity of the Frankl-family radii}\label{sec:numerics}

Throughout this section, the standing assumption is $n>2k-t$.

We begin by making the Frankl-family radii explicit.  This orbit reduction is
used only for numerical comparisons; the later support-reduction argument
does not assume that an arbitrary intersecting family has comparable
symmetry.

\begin{proposition}\label{prop:candidate-orbit-formula}
For every $0\le r\le k-t$, the spectral radius of the Frankl family $\A_r$ is
\begin{equation}\label{eq:candidate-orbit-formula}
 \rho_r=
 \max_{0\le u\le1}
 k\sum_{j=t+r}^{k}
 \binom{t+2r}{j}\binom{n-t-2r}{k-j}
 \left(\frac{u}{t+2r}\right)^{j/k}
 \left(\frac{1-u}{n-t-2r}\right)^{(k-j)/k}.
\end{equation}
Here and below, an inadmissible binomial coefficient is zero, and a factor
with exponent zero is interpreted as $1$.
\end{proposition}

\begin{proof}
Let $y_i=x_i^k$.  For every edge $F$, the corresponding monomial becomes
$\prod_{i\in F}y_i^{1/k}$, the geometric mean of the $k$ coordinates indexed
by $F$, and is therefore concave on the nonnegative orthant.  Hence
$P_{\A_r}$, expressed in the variables $y$, is concave.  It is also invariant
under permutations within the core $[t+2r]$ and within its complement.
Averaging $y$ over this automorphism group and applying concavity cannot
decrease the polynomial.  Thus some maximizing vector has constant $k$th
powers on each of the two vertex orbits.

Let
\[
 u:=\sum_{i=1}^{t+2r}x_i^k.
\]
The common $k$th powers on the core and its complement are then $u/(t+2r)$ and
$(1-u)/(n-t-2r)$, respectively.  There are
$\binom{t+2r}{j}\binom{n-t-2r}{k-j}$ edges having exactly $j$ core vertices.
Summing their contributions and maximizing over $0\le u\le1$ gives
\eqref{eq:candidate-orbit-formula}.
\end{proof}

This section isolates the numerical content of the Frankl-family sequence.
We first use a lower bound on the universal pull gain to construct a boundary pull from  $\A_r$ to $\A_{r-1}$ that
does not decrease spectral radius.  We then prove that for $t\ge2$ the single
condition $\rho_0\ge\rho_1$ propagates through the whole list of Frankl
families.  This is the numerical half of
Theorem~\ref{thm:first-candidate}.  The boundary-rank
bookkeeping used in the structural pull is placed in Section~\ref{sec:proof},
and the case $t=1$ is postponed to Section~\ref{sec:t-one}.

For the proof only, it is convenient to name domination of all the Frankl
families.  We say that the full $t$-star \emph{dominates the Frankl families}
if
\begin{equation}\label{eq:candidate-dominance}
 \rho_0\ge\rho_r\qquad\text{for every }1\le r\le k-t.
\end{equation}
This auxiliary notion will be eliminated from the statement of the main
theorem by Proposition~\ref{prop:propagation}.

Dominance over the Frankl families has two immediate numerical consequences.

\begin{lemma}\label{lem:numerical-consequences}
Assume $1\le t<k$ and that the full $t$-star dominates the Frankl families.
Then
\begin{equation}\label{eq:ground-set-gap}
 n\ge2k-t+2
\end{equation}
and
\begin{equation}\label{eq:basic-numerical-condition}
 n-k-1\ge t(k-t)
 \left(\frac{k-t}{n-t}\right)^{1/k}.
\end{equation}
\end{lemma}

\begin{proof}
Suppose first that $n=2k-t+1$.  Fixing $k-t$ and comparing the two endpoint
values in~\eqref{eq:endpoint-radii}, at the smallest possible value
$k=(k-t)+1$ we have
\[
 \frac{\rho_{k-t}}{\rho_0}
 =\frac{k-t+1}{2(k-t)+1}
  \left(\frac{2(k-t)+1}{k-t}\right)^{(k-t)/(k-t+1)}>1,
\]
because this inequality is equivalent to
$((k-t)+1)^{k-t+1}>(2(k-t)+1)(k-t)^{k-t}$, which follows from
\[
 \left(1+\frac1{k-t}\right)^{k-t}\ge2>
 \frac{2(k-t)+1}{k-t+1}.
\]
Under the replacement
\[
(k,t,n)\longmapsto(k+1,t+1,n+1),
\]
the ratio $\rho_{k-t}/\rho_0$ is multiplied by
\[
\left(1+\frac{k-t}{k}\right)
\left(\frac{k-t}{2(k-t)+1}\right)^{(k-t)/[k(k+1)]}>1.
\]
Indeed, after raising both sides to the power $k(k+1)$, it is enough to show
\[
\left(1+\frac{k-t}{k}\right)^{k(k+1)}
>
\left(\frac{2(k-t)+1}{k-t}\right)^{k-t}.
\]
By Bernoulli's inequality,
\[
\left(1+\frac{k-t}{k}\right)^{k(k+1)}
\ge (k-t+1)^{k+1}.
\]
Moreover,
\[
(k-t+1)^{k+1}>3^{k-t}
\ge
\left(\frac{2(k-t)+1}{k-t}\right)^{k-t}.
\]
Note that the first strict inequality is immediate when $k-t=1$, while for
$k-t\ge2$ it follows from $k+1\ge k-t+2$ and $k-t+1\ge3$.

Hence $\rho_{k-t}>\rho_0$ whenever $n=2k-t+1$, contradicting dominance over
the Frankl families.  This proves~\eqref{eq:ground-set-gap}.

For~\eqref{eq:basic-numerical-condition}, evaluate $\A_1$ at a Perron vector
of the star.  Its first $t$ coordinates  have a common value $a$, the other
$n-t$ coordinates have a common value $b$, and
\[
 \frac ba=\left(\frac{k-t}{n-t}\right)^{1/k}.
\]
Passing from $\A_0$ to $\A_1$, the followings hold: \begin{align*}
\A_0\setminus\A_1&=\{F\in\binom{[n]}k:F\cap[t+2]=t\}\\
\A_1\setminus\A_0&=\{F\in\binom{[n]}k:F\cap[t]\in\binom{[t]}{t-1},[t+1,t+2]\subseteq F\}.
\end{align*}  Since $\rho_1\le\rho_0$, evaluation at
this vector gives
\[
 t\binom{n-t-2}{k-t-1}a^{t-1}b^{k-t+1}
 \le \binom{n-t-2}{k-t}a^tb^{k-t}.
\]
Cancellation yields $t(k-t)(b/a)\le n-k-1$, which is exactly
\eqref{eq:basic-numerical-condition}.
\end{proof}

The same necessary estimate explains the leading term in
Corollary~\ref{cor:explicit-threshold}.  Fix $t$ and suppose that
$n=c(k-t)+O_t(1)$ as $k-t\to\infty$.  Then
\eqref{eq:basic-numerical-condition}, divided by $k-t$, becomes
\[
 c-1+o(1)
 \ge t\left(\frac1{c+o(1)}\right)^{1/k}
 =t+o(1).
\]
Consequently $c\ge t+1$. Thus no threshold of the form with a smaller coefficient of $k-t$ can be valid for all sufficiently large  $k$.

The next lemma is the numerical estimate needed for adjacent Frankl-family
comparisons; it will also be reused in the diagonal stage of the boundary
pull.

\begin{lemma}\label{lem:middle-envelope}
Let $t\ge2$.  Assume~\eqref{eq:ground-set-gap} and
\eqref{eq:basic-numerical-condition} hold.  For $1\le r\le k-t$, let
\[
 L_{t,r}:=\min\left\{r,\frac{t+2r-1}{r}\right\}.
\]
Then the corresponding universal gain in Lemma~\ref{lem:universal-gain}
satisfies
\begin{equation}\label{eq:middle-envelope}
 g_{k-t-r}>L_{t,r}.
\end{equation}
\end{lemma}

\begin{proof}
We separate $t\ge3$ and $t=2$.

\smallskip
\noindent\emph{Case 1: $t\ge3$.}
Let
\[
 \tau:=\frac{k}{k-1}.
\]
When $k-t-r>0$, the two factors in~\eqref{eq:universal-gain} satisfy
\[
 \begin{aligned}
 &\frac{n-t-2r+1}{k-t-r+1}
 =2+\frac{n-(2k-t)-1}{k-t-r+1}
 \ge2+\frac{n-(2k-t)-1}{k-t}=1+\frac{n-k-1}{k-t},\\
 &\frac{n-t-2r}{k-t-r}
 >\frac{n-t-2r+1}{k-t-r+1}.
 \end{aligned}
\]
The second comparison is strict and has positive exponent, so the resulting
lower bound for $g_{k-t-r}$ is strict.  When $k-t-r=0$,
\[
 n-t-2r+1-\left(1+\frac{n-k-1}{k-t}\right)
 =\frac{(k-t-1)(n-(2k-t)-1)}{k-t},
\]
so the same strict bound follows from~\eqref{eq:gain-q-zero} unless
$k-t=r=1$;
in that exceptional
subcase $L_{t,r}=1$ and $g_0>1$ directly.  Hence, apart from this already settled
subcase,
\begin{equation}\label{eq:g-H-bound}
 g_{k-t-r}>\left(1+\frac{n-k-1}{k-t}\right)^{(k-1)/k}.
\end{equation}
Condition~\eqref{eq:basic-numerical-condition} is equivalent to
\begin{equation}\label{eq:H-condition}
 \frac{n-k-1}{k-t}\left(\frac{n-t}{k-t}\right)^{1/k}\ge t.
\end{equation}
Suppose~\eqref{eq:middle-envelope} does not hold. Then~\eqref{eq:g-H-bound} would give
$\frac{n-k-1}{k-t}<L_{t,r}^\tau-1$.  Since
\[
 \frac{n-t}{k-t}=1+\frac{n-k-1}{k-t}+\frac1{k-t}<L_{t,r}^{\tau}+1,
\]
~\eqref{eq:H-condition} would imply
\begin{equation}\label{eq:Phi-contradiction}
 t<\Phi:=(L_{t,r}^\tau-1)(L_{t,r}^\tau+1)^{1/k}.
\end{equation}

Assume first that $L_{t,r}=r$, equivalently $(r-1)^2\le t$.  Since
$k-1\ge t+r-1\ge r(r-1)$, one has $r^\tau\le r+1$.  For $r\ge3$ this follows
from
\[
 (k-1)\log(1+1/r)\ge\frac{r(r-1)}{r+1}\ge\log r;
\]
the $r=1,2$ cases are immediate.  Therefore
$
 \Phi\le r(r+2)^{1/k}.
$
For $r=1$ this is already less than $t$.  For $r=2$ it is at most
$2\cdot4^{1/5}<3\le t$; for $r=3$ it is at most
$3\cdot5^{1/7}<4\le t$; and for $r\ge4$ it is less than
$2r\le(r-1)^2\le t$.  Each alternative contradicts
\eqref{eq:Phi-contradiction}.

Now assume
$L_{t,r}=2+(t-1)/r$, equivalently $(r-1)^2\ge t$.  Let $a=\sqrt t$. So
$r\ge a+1$ and $L_{t,r}\le a+1$.  We claim
$L_{t,r}^\tau\le L_{t,r}+1$.  It is enough to show
$
 \log L_{t,r}\le(t+r-1)\log(1+1/L_{t,r}).
$
When $t\ge4$, use
$\log(1+1/L_{t,r})\ge1/(L_{t,r}+1)$ and
\begin{equation}\label{eq:add}
 \log(a+1)\le\frac{a(a+1)}{a+2}.
\end{equation}
Note that \eqref{eq:add} holds since the difference between the right- and left-hand sides of the last inequality
is positive at $a=2$ and has derivative
\[
 \frac{a^3+4a^2+2a-2}{(a+1)(a+2)^2}>0.
\]
Since $k\ge a^2+a+1$ and
$(a+3)^{1/(a^2+a+1)}\le4/3$, we obtain
\[
 \Phi\le L_{t,r}(L_{t,r}+2)^{1/k}
 \le\frac43(a+1)\le a^2=t,
\]
 a contradiction with \eqref{eq:Phi-contradiction}.  If $t=3$, then $r\ge3$ by $(r-1)^2\ge t$.
The expression $\Phi$ increases with $L_{t,r}$ and decreases with $k$: for fixed
$L_{t,r}>1$, both $L_{t,r}^{k/(k-1)}$ and the outer exponent $1/k$ decrease
with $k$.  Hence the worst case is $r=3$, $k=6$, and $L_{3,3}=8/3$.  Since
$(8/3)^{6/5}<13/4$,
\[
 \Phi<\frac94\left(\frac{17}{4}\right)^{1/6}<3=t,
\] a contradiction with \eqref{eq:Phi-contradiction}.

\smallskip
\noindent\emph{Case 2: $t=2$.}
First, we show that ~\eqref{eq:basic-numerical-condition} and
\eqref{eq:ground-set-gap} imply
\begin{equation}\label{eq:t-two-gap}
 n-(2k-t)\ge k-t-1.
\end{equation}
This is immediate for $k-t\le3$.  If $k-t\ge4$ and
$n-(2k-t)\le k-t-2$, the left side of
\eqref{eq:basic-numerical-condition} is at most $2(k-t)-3$, whereas the
right side is greater than $2(k-t)3^{-1/k}$.  But
\[
 3^{-1/k}>1-\frac{3}{2(k-t)},
\]
because $(1-3/(2(k-t)))^k<e^{-3/2}<1/3$, a contradiction.

If $r=1$, then $L_{2,1}=1$ and the desired strict lower bound follows from
$g_{k-t-r}>1$.  When $r=2$, $L_{2,2}=2$.  If $k-t-r=0$,
then $g_0\ge3^{3/4}>2$.  When $k-t-r\ge1$,  applying
\eqref{eq:ground-set-gap} to~\eqref{eq:universal-gain} yields
\[
 \begin{aligned}
 g_{k-t-r}^k
 &\ge
 \left(2+\frac1{k-t-r+1}\right)^3
 \left(2+\frac2{k-t-r}\right)^{k-t-r}\\
 &\ge2^{k-1}\left(1+\frac1{2(k-t-r)+2}\right)
              \left(1+\frac1{k-t-r}\right)^{k-t-r}\\
 &>2^k.
 \end{aligned}
\]
Finally, if $r\ge3$,~\eqref{eq:t-two-gap} gives
\[
 \frac{n-t-2r+1}{k-t-r+1}\ge3,
\]
and hence
\[
 g_{k-t-r}\ge3^{(k-1)/k}
 \ge3^{4/5}>\frac73\ge2+\frac1r=L_{2,r}.
\]

This completes the proof.
\end{proof}

We next translate the gain estimate into a combinatorial operation on a
Frankl family.  The relevant condition ensures that the sum of the values
$P_{\C_{s_r-1}(A\setminus\{s_r\})}(y)$ over a suitable boundary
subcollection is at least the value of $P$ on the old boundary layer.

\begin{lemma}\label{lem:adjacent-pull}
Fix $1\le r\le k-t$, put $s_r=t+2r$, and let
$\G_r=\binom{[s_r]}{t+r}$ be the minimal generating antichain of
$\A_r$.  For $a\in[s_r-1]$, define
\[
 \mathcal H_{r,a}:=
 \{G\in\G_r:s_r\notin G\}\cup
 \{G\setminus\{s_r\}:G\in\G_r,\ s_r\in G,\ a\notin G\},
 \qquad
 \mathcal Q_{r,a}:=\bigcup_{H\in\mathcal H_{r,a}}\D(H).
\]
Thus $\mathcal Q_{r,a}$ is obtained by trying to remove the last core
coordinate $s_r$: among the boundary generators, it pulls precisely those
omitting $a$ and discards the others.  If
\begin{equation}\label{eq:adjacent-pull-condition}
 g_{k-t-r}\ge B_{t,r}:=\frac{t+2r-1}{r}
\end{equation}
holds, then there is a coordinate $a\in[s_r-1]$ such that
$\mathcal Q_{r,a}$ is a copy of $\A_{r-1}$ and
\[
 \rho(\mathcal Q_{r,a})\ge\rho(\A_r).
\]
Moreover, if~\eqref{eq:adjacent-pull-condition} is strict, then
$\rho_{r-1}>\rho_r$.
\end{lemma}

\begin{proof}
Here the gain $g_{k-t-r}$ is computed with $n-t-2r$ tail coordinates.
The minimal generators of $\A_r$ are all $(t+r)$-subsets of its core
$[t+2r]$.
Let $x$ be a Perron vector of $\A_r$ after the tail equalization above, and
let $y$ be defined from $x$ by~\eqref{eq:tail-equalization}, with $s=s_r$.
Put
\[
 \G_r^{\partial}:=\{A\in\G_r:s_r\in A\},
 \qquad
 \mathcal B_r^{\partial}
 :=\dotcupbig_{A\in\G_r^{\partial}}\C_{s_r}(A),
 \qquad
 \mathcal B_{r,a}^{\partial}
 :=\dotcupbig_{\substack{A\in\G_r^{\partial}\\a\notin A}}\C_{s_r}(A).
\]
Every $A\in\G_r^{\partial}$ omits exactly $r$ points of $[s_r-1]$.  Since
the families $\C_{s_r}(A)$ are pairwise disjoint, linearity of $P$ gives
\[
 \sum_{a=1}^{s_r-1}
 P_{\mathcal B_{r,a}^{\partial}}(x)
 =rP_{\mathcal B_r^{\partial}}(x).
\]
Choose $a\in[s_r-1]$ such that
\begin{equation}\label{eq:adjacent-selected-contribution}
 P_{\mathcal B_{r,a}^{\partial}}(x)
 \ge\frac{r}{s_r-1}P_{\mathcal B_r^{\partial}}(x).
\end{equation}

By definition, $\mathcal H_{r,a}$ deletes the entire boundary layer and
replaces each boundary generator $A$ omitting $a$ by
$A\setminus\{s_r\}$.  These inserted sets are precisely all
$(t+r-1)$-subsets of $[s_r-1]\setminus\{a\}$.  Every old nonboundary generator
contains one of the newly inserted sets, so after deleting redundant generators the resulting
family is exactly $\A_{r-1}$ on the core
$[s_r-1]\setminus\{a\}$.

Set
\[
 \mathcal N_{r,a}
 :=\dotcupbig_{\substack{A\in\G_r^{\partial}\\a\notin A}}
   \C_{s_r-1}(A\setminus\{s_r\}).
\]
Formula~\eqref{eq:gain-factor} and
Lemma~\ref{lem:universal-gain} give
\[
 P_{\mathcal N_{r,a}}(y)
 =\Gamma_{k-t-r}(\eta)P_{\mathcal B_{r,a}^{\partial}}(x)
 \ge g_{k-t-r}P_{\mathcal B_{r,a}^{\partial}}(x).
\]
Therefore~\eqref{eq:adjacent-selected-contribution} and
\eqref{eq:adjacent-pull-condition} imply
\[
 P_{\mathcal N_{r,a}}(y)
 \ge \frac{r}{s_r-1}g_{k-t-r}
       P_{\mathcal B_r^{\partial}}(x)
 \ge P_{\mathcal B_r^{\partial}}(x).
\]
Together with Lemmas~\ref{lem:trace-decomposition} and
\ref{lem:tail-symmetrization}, this yields
$P_{\mathcal Q_{r,a}}(y)\ge P_{\A_r}(x)=\rho_r$.  Since
$\mathcal Q_{r,a}$ is a copy of $\A_{r-1}$, it follows that
$\rho_{r-1}\ge\rho_r$.  Moreover,
$P_{\mathcal B_r^{\partial}}(x)>0$ by the
Perron--Frobenius comparison above, so strictness in
\eqref{eq:adjacent-pull-condition} gives $\rho_{r-1}>\rho_r$.
\end{proof}

We can now prove the propagation statement that reduces the main theorem to
the first comparison $\rho_0\ge\rho_1$.

\begin{proposition}\label{prop:propagation}
Let $2\le t<k$ and $n>2k-t$.  If $\rho_0\ge\rho_1$, then
\begin{equation}\label{eq:candidate-sequence}
 \rho_0\ge\rho_1\ge\cdots\ge\rho_{k-t}.
\end{equation}
\end{proposition}

\begin{proof}
The comparison $\rho_0\ge\rho_1$ alone suffices to yield
\eqref{eq:basic-numerical-condition}; this follows from the second half of the proof of
Lemma~\ref{lem:numerical-consequences}.  It also forces
\eqref{eq:ground-set-gap}.  Indeed, if $n=2k-t+1$, then
\eqref{eq:basic-numerical-condition} becomes
\[
 1\ge t\left(\frac{k-t}{2(k-t)+1}\right)^{1/k},
\]
whereas
\[
 t^k(k-t)\ge2^{k-t+2}(k-t)>2(k-t)+1.
\]

Fix $2\le r\le k-t$, and let
\[
 B_{t,r}=2+\frac{t-1}{r}.
\]
We prove $g_{k-t-r}\ge B_{t,r}$.  If $B_{t,r}\le r$, this follows from
Lemma~\ref{lem:middle-envelope}.  We may therefore assume
\begin{equation}\label{eq:B-large}
 B_{t,r}>r,
 \qquad\text{equivalently}\qquad t>(r-1)^2.
\end{equation}
Let
\[
 \qquad \kappa:=\frac{k}{k-1}.
\]
The estimates~\eqref{eq:g-H-bound} and~\eqref{eq:H-condition} use only
\eqref{eq:ground-set-gap} and~\eqref{eq:basic-numerical-condition}.  Suppose $g_{k-t-r}<B_{t,r}$. Then
\begin{equation}\label{eq:Phi-B}
 t<\Phi_{t,r,k}:=(B_{t,r}^\kappa-1)(B_{t,r}^\kappa+1)^{1/k}.
\end{equation}

Assume first that $r\ge3$.  Condition~\eqref{eq:B-large} gives $t\ge5$,
$B_{t,r}\le(t+5)/3$, and $k\ge t+3$.  The elementary inequalities
\[
 \left(\frac{t+5}{3}\right)^{1/(t+2)}\le\frac65,
 \qquad
 \left(\frac{2t+15}{5}\right)^{1/(t+3)}\le\frac54
 \qquad(t\ge5)
\]
hold because the logarithms of the right-hand sides minus the left-hand sides
are positive at $t=5$, and their derivatives are respectively
$\log(6/5)-1/(t+5)>0$ and
$\log(5/4)-2/(2t+15)>0$.  Therefore
\[
 B_{t,r}^\kappa\le\frac{2(t+5)}5,
 \qquad
 \Phi_{t,r,k}\le\frac{2t+5}{5}\cdot\frac54
 =\frac{2t+5}{4}\le t,
\]
a contradiction with~\eqref{eq:Phi-B}.

Now let $r=2$.  If $t\ge7$, then $B_{t,r}=(t+3)/2$, $k\ge t+2$, and
\[
 \left(\frac{t+3}{2}\right)^{1/(t+1)}\le\frac54,
 \qquad
 \left(\frac{5t+23}{8}\right)^{1/(t+2)}\le\frac43.
\]
Indeed, after taking logarithms, both inequalities hold at $t=7$ and remain
valid because
$\log(5/4)-1/(t+3)>0$ and
$\log(4/3)-5/(5t+23)>0$.
It follows that
\[
 \Phi_{t,2,k}\le\frac{5t+7}{8}\cdot\frac43
 =\frac{5t+7}{6}\le t,
\]
a contradiction with~\eqref{eq:Phi-B}.
For $t=6$, writing $B_{t,r}=9/2$ and using $k\ge8$ gives
\[
 B_{t,r}^\kappa\le B_{t,r}^{8/7}<\frac{28}{5},
 \qquad
 (B_{t,r}^\kappa+1)^{1/k}<\left(\frac{33}{5}\right)^{1/8}<\frac{13}{10},
\]
Thus $\Phi_{6,2,k}<(23/5)(13/10)<6=t$, a contradiction
with~\eqref{eq:Phi-B}.  The remaining pairs are
\[
 (t,r)\in\{(2,2),(3,2),(4,2),(5,2)\}.
\]

For these four pairs, $B_{t,r}=(t+3)/2$.  If $k-t-2=0$, then
\eqref{eq:basic-numerical-condition} forces
$n-(2k-t)\ge t+1$.  Indeed, if $n-(2k-t)\le t$, then
\[
 n-(2k-t)+1
 \ge2t\left(\frac2{4+n-(2k-t)}\right)^{1/(t+2)}
 \ge2t\left(\frac2{t+4}\right)^{1/(t+2)}>t+1,
\]
contradicting $n-(2k-t)+1\le t+1$.  Hence
\[
 g_0=(n-(2k-t)+1)^{(t+1)/(t+2)}
 \ge(t+2)^{(t+1)/(t+2)}>B_{t,r}.
\]
If $k-t-2=1$, the same argument forces
$n-(2k-t)\ge2t-1$.  Indeed, if $n-(2k-t)\le2t-2$, then
\[
 n-(2k-t)+2
 \ge3t\left(\frac3{6+n-(2k-t)}\right)^{1/(t+3)}
 \ge3t\left(\frac3{2t+4}\right)^{1/(t+3)}>2t.
\]
This contradicts $n-(2k-t)+2\le2t$.  We then obtain
\[
 g_1^{t+3}
 =\left(\frac{n-(2k-t)+3}{2}\right)^{t+1}
  \bigl(n-(2k-t)+2\bigr)
 \ge(t+1)^{t+1}(2t+1)>B_{t,r}^{t+3}.
\]

It remains to consider $k-t-2\ge2$.  For $2\le t\le5$,
\eqref{eq:basic-numerical-condition} implies
\begin{equation}\label{eq:small-t-gap}
 n-(2k-t)\ge(t-1)(k-t-2)+1.
\end{equation}
Indeed, if $n-(2k-t)\le(t-1)(k-t-2)$, the left side of
\eqref{eq:basic-numerical-condition} is at most $t(k-t-2)+1$, whereas its
right side is greater than
\[
 t(k-t)
 \left(\frac{k-t}{(t+1)(k-t-2)+4}\right)^{1/k}
 >t(k-t)(t+1)^{-1/k}>t(k-t-2)+1.
\]
For the last step use $e^{-u}>1-u$ and
$t\log(t+1)<2t-1$ for $2\le t\le5$.

From~\eqref{eq:small-t-gap},
\[
 g_{k-t-2}^{k}\ge
 \left(t+1-\frac{t-1}{k-t-1}\right)^{t+1}
 \left(t+1+\frac1{k-t-2}\right)^{k-t-2}.
\]
For an integer $u\ge2$, the ratio
\[
 \frac{\left(t+1-\frac{t-1}{u+1}\right)^{t+1}
       \left(t+1+\frac1u\right)^u}
      {B_{t,r}^{t+u+2}}
\]
increases with $u$: the first factor in the numerator increases, while
\[
 \frac{(t+1+1/u)^u}{B_{t,r}^u}
 =\left(\frac{t+1}{B_{t,r}}\right)^u
  \left(1+\frac1{(t+1)u}\right)^u
\]
also increases.  For $t=3,4,5$ the displayed ratio already exceeds $1$ at
$u=2$, as witnessed respectively by
\[
 \left(\frac{10}{3}\right)^4\left(\frac92\right)^2>3^7,
 \qquad
 4^5\left(\frac{11}{2}\right)^2>\left(\frac72\right)^8,
 \qquad
 \left(\frac{14}{3}\right)^6\left(\frac{13}{2}\right)^2>4^9.
\]
For $t=2$, the same monotonicity starts at $u=3$, where
\[
 \left(\frac{11}{4}\right)^3\left(\frac{10}{3}\right)^3>
 \left(\frac52\right)^7.
\]
In the remaining case $t=2$ and $k-t-2=2$,
\eqref{eq:basic-numerical-condition} gives $n-(2k-t)\ge4$ (the value $3$
would require $6\ge8(4/11)^{1/6}>6$).  Consequently,
\[
 g_2^6\ge3^3 4^2=432>\left(\frac52\right)^6.
\]
Thus $g_{k-t-r}\ge B_{t,r}$ in every case.
Lemma~\ref{lem:adjacent-pull} gives
$\rho_{r-1}\ge\rho_r$, and iteration proves~\eqref{eq:candidate-sequence}.
\end{proof}

\section{The support-reducing boundary pull and the Frankl-family comparison principle}
\label{sec:proof}

Retain the pair $(\F,x)$, its minimal generating antichain $\G$, and
$s=\sigma(\G)$ from Section~\ref{sec:generators}.  We use throughout the
single pull notation $\mathcal H(\mathcal S),\mathcal Q(\mathcal S)$ defined
in~\eqref{eq:macro-generators}.  The aim of this section is now explicit:
choose $\mathcal S\subseteq\G^{\partial}$ so that
$\mathcal H(\mathcal S)$ is $t$-intersecting and
$\rho(\mathcal Q(\mathcal S))\ge\rho(\F)$.  Since every member of
$\mathcal H(\mathcal S)$ lies in $[s-1]$, such a choice lowers the support
parameter from $s$ to at most $s-1$.

By Lemma~\ref{lem:tight-pairs}, the boundary ranks occur in complementary
pairs $p,q$ satisfying $p+q=s+t$.  For $p<q$, put
$\mathcal S_p=\G_p^{\partial}$ and $\mathcal S_q=\G_q^{\partial}$.
Lemma~\ref{lem:profitable-off-middle} gives an $\ell\in\{p,q\}$ satisfying
\eqref{eq:off-middle-polynomial}.  If $p=q$, that orbit contains only one rank;
Lemmas~\ref{lem:frequent-omission} and~\ref{cor:profitable-middle} select a
subcollection of its shadows with a common omitted coordinate.  Finally,
Lemmas~\ref{lem:macro-intersection} and~\ref{lem:spectral-ledger} combine all
these local choices and verify the two requirements above.

Throughout the section, assume that the full $t$-star dominates the Frankl
families.  Whenever a diagonal boundary rank has parameter $r$, we additionally
assume
\begin{equation}\label{eq:candidate-separation}
 \rho(\F)>\rho_r.
\end{equation}
This holds in both applications below: in the value argument
$\rho(\F)>\rho_0\ge\rho_r$, while in the rigidity argument
$\rho(\F)=\rho_0>\rho_r$.  Dominance over the Frankl families also gives
$n\ge2k-t+2$ by Lemma~\ref{lem:numerical-consequences}.

\subsection{The off-diagonal boundary pull}\label{subsec:off-diagonal}

The involution on boundary ranks supplied by~\eqref{eq:rank-pairing} is
$i\mapsto s+t-i$.  We first treat one of its two-element orbits.

For a boundary layer of rank $k-q$, let
\[
 \mathcal B_q:=\mathcal B(\G_{k-q}^{\partial}),
 \qquad
 \Lambda_q(\eta):=\Gamma_q(\eta)-1.
\]
\begin{lemma}\label{lem:profitable-off-middle}
Let $p<q$ be boundary ranks with
\[
 p+q=s+t.
\]
Put $\mathcal S_p:=\G_p^{\partial}$ and
$\mathcal S_q:=\G_q^{\partial}$.  Then, for some $\ell\in\{p,q\}$,
\begin{equation}\label{eq:off-middle-polynomial}
 P_{\mathcal N(\mathcal S_\ell)}(y)
 =\Gamma_{k-\ell}(\eta)P_{\mathcal B_{k-\ell}}(x)
 \ge P_{\mathcal B_{k-p}}(x)+P_{\mathcal B_{k-q}}(x).
\end{equation}
If
$P_{\mathcal B_{k-p}}(x)+P_{\mathcal B_{k-q}}(x)>0$, then
\eqref{eq:off-middle-polynomial} is strict.
\end{lemma}

\begin{proof}
Put $u:=k-p$, $v:=k-q$, and $N:=n-s$.  Then $u>v\ge0$ and
Lemma~\ref{lem:tight-pairs} gives
\begin{equation}\label{eq:off-diagonal-parameters}
 u+v=2k-s-t,
 \qquad N\ge u+v+2.
\end{equation}
We first prove the numerical inequality
\begin{equation}\label{eq:pair-product}
 (g_u-1)(g_v-1)>1.
\end{equation}
The only formal exception to this inequality is
$(k,u,v,N)=(2,1,0,3)$; it would give complementary boundary generators of
sizes $1$ and $2$, forcing $t=1$, $s=2$, and the two generators
$\{2\}$ and $[2]$, contrary to the antichain property.

Let $c_j=1/g_j$.  Inequality~\eqref{eq:pair-product} is equivalent to
\begin{equation}\label{eq:c-sum}
 c_u+c_v<1,
\end{equation}
where
\[
 c_j=\left(\frac jN\right)^{j/k}
     \left(\frac{j+1}{N+1}\right)^{(k-j-1)/k},
\]
with the first factor interpreted as $1$ for $j=0$.  For fixed $j$, this
quantity decreases with $N$.  It also decreases with $k$; for $j=0$ this is
immediate.  For $j\ge1$, the latter assertion is equivalent, with $a=j/N$
and $b=(j+1)/(N+1)$, to $a^j\ge b^{j+1}$, or
\[
 \frac{(N+1)^{j+1}}{N^j}\ge\frac{(j+1)^{j+1}}{j^j};
\]
this follows because $(x+1)^{j+1}/x^j$ increases for $x\ge j$.  The worst
case is therefore $N=u+v+2$ and $k=u+1$.

In this case, weighted H\"older gives
\[
 \begin{aligned}
 c_u+c_v
 &\le 2^{1/(u+1)}
 \left(\frac{u+v}{N}\right)^{v/(u+1)}
 \left(\frac uN+\frac{v+1}{N+1}\right)^{(u-v)/(u+1)}\\
 &=2^{1/(u+1)}
 \left(1-\frac2N\right)^{v/(u+1)}
 \left(1-\frac{u+2v+4}{N(N+1)}\right)^{(u-v)/(u+1)}.
 \end{aligned}
\]
Using $\log(1-x)\le-x$, the logarithm of the right-hand side, multiplied by
$u+1$, is at most $\log2-E(u,v)$, where
\[
 E(u,v):=\frac{2v}{N}
 +\frac{(u-v)(u+2v+4)}{N(N+1)}.
\]
A direct simplification gives
\[
 E(u,v)-E(u,0)
 =-\frac{v\bigl(vu^2+4vu-u^3-4u^2-8u-12\bigr)}
 {(u+2)(u+3)(u+v+2)(u+v+3)}\ge0,
\]
because $v\le u-1$.  Hence, for $u\ge3$,
\[
 E(u,v)\ge E(u,0)=\frac{u(u+4)}{(u+2)(u+3)}
 \ge\frac7{10}>\log2,
\]
which proves~\eqref{eq:c-sum}.  For $u=2$, the cases $v=0,1$ reduce to
\[
 2^{-2/3}+5^{-2/3}<1,
 \qquad
 (2/5)^{2/3}+15^{-1/3}<1.
\]
For $u=1,v=0$, monotonicity in $k$ and $N$ reduces all nonexceptional cases
to $(k,N)=(2,4)$ or $(3,3)$, where
\[
 \frac12+\frac1{\sqrt5}<1,
 \qquad
 6^{-1/3}+4^{-2/3}<1.
\]
This proves~\eqref{eq:pair-product}.

Since $\Gamma_j(\eta)\ge g_j$, the numerical inequality gives
\[
 \Lambda_u(\eta)\Lambda_v(\eta)>1.
\]
It follows that at least one of
\begin{equation}\label{eq:off-middle-ledger}
 \Lambda_u(\eta)P_{\mathcal B_u}(x)-P_{\mathcal B_v}(x)\ge0,
 \qquad
 \Lambda_v(\eta)P_{\mathcal B_v}(x)-P_{\mathcal B_u}(x)\ge0
\end{equation}
holds.  Indeed, if
$P_{\mathcal B_u}(x),P_{\mathcal B_v}(x)>0$ and both quantities were
nonpositive, multiplication would contradict the product inequality.
Therefore at least one of them is strictly positive.  If exactly one of the
two polynomial values is positive, select its layer; since
$\Gamma_j(\eta)>1$, the corresponding inequality is again strict.  If both
values are zero, either choice works.  Selecting the $p$-rank when the first
quantity is nonnegative or the $q$-rank when the second is nonnegative gives
\eqref{eq:off-middle-polynomial}, with the asserted strictness.
\end{proof}

Thus, from every two-element orbit of the boundary-rank involution, one whole
rank can be selected so that~\eqref{eq:off-middle-polynomial} holds.  The only
remaining orbit is a fixed point, which exists exactly when the middle rank
below is integral.

\subsection{The diagonal stage}\label{subsec:diagonal}

The remaining fixed point exists only when $s-t$ is even.  In that case let
\begin{equation}\label{eq:diagonal-parameters}
 r:=\frac{s-t}{2}.
\end{equation}
Then $1\le r\le k-t$, and the corresponding boundary deficit is
$k-t-r$.  For
$E\in\G_{t+r}^{\partial}$ define its omitted set
\[
 C_E:=[s-1]\setminus(E\setminus\{s\}),
 \qquad |C_E|=r.
\]
For each $a\in[s-1]$, the corresponding candidate output pulls exactly the
middle-rank generators omitting $a$:
\[
 \mathcal S_a'
 :=\{E\in\G_{t+r}^{\partial}:a\in C_E\},
 \qquad
 \mathcal H_a'
 :=\G^{\circ}\cup
   \{E\setminus\{s\}:E\in\mathcal S_a'\},
 \qquad
 \mathcal Q_a'
 :=\bigcup_{H\in\mathcal H_a'}\D(H).
\]
Define
\[
 \mathcal B'
 :=\mathcal B(\G_{t+r}^{\partial}),
 \qquad
 \mathcal B_a'
 :=\mathcal B(\mathcal S_a'),
 \qquad
 \mathcal N_a'
 :=\mathcal N(\mathcal S_a').
\]
Whenever $P_{\mathcal B'}(x)>0$, put
\begin{equation}\label{eq:delta-definition}
 \delta_a
 :=\frac{P_{\mathcal B_a'}(x)}
          {P_{\mathcal B'}(x)},
 \qquad a\in[s-1].
\end{equation}

\begin{lemma}\label{lem:frequent-omission}
Assume that the full $t$-star dominates the Frankl families and
$\rho(\F)>\rho_r$, where $r$ is given by
\eqref{eq:diagonal-parameters}.  If
$P_{\mathcal B'}(x)>0$,
then there is a coordinate $a_0\in[s-1]$ such that the candidate generating
family $\mathcal H_{a_0}'$ satisfies
\begin{equation}\label{eq:delta-bound}
 \delta_{a_0}=\max_{a\in[s-1]}\delta_a
 \ge\max\left\{\frac r{s-1},\frac1r\right\}.
\end{equation}
\end{lemma}

\begin{proof}
Since every $C_E$ has size $r$ and the families $\C_s(E)$ are pairwise
disjoint,
\[
 \sum_{p=1}^{s-1}\delta_p
 =\frac{1}{P_{\mathcal B'}(x)}
   \sum_{E\in\G_{t+r}^{\partial}}
   |C_E|P_{\C_s(E)}(x)
 =r.
\]
This gives the first term in~\eqref{eq:delta-bound}.

Suppose first that an off-middle boundary rank occurs.  Take a generator
$A\in\G_{t+a}^{\partial}$ from the smaller member of a complementary pair,
where $1\le a<r$, and let
\[
 T=A\setminus\{s\},
 \qquad Q=[s-1]\setminus T.
\]
For a middle-rank generator $E$, equality $|A\cap E|=t$ would force
$|A|+|E|=s+t$ by Lemma~\ref{lem:tight-pairs}, and hence $a=r$, a
contradiction.  Thus $|T\cap(E\setminus\{s\})|\ge t$, so
\[
 |C_E\cap Q|\ge r-a+1,
 \qquad |Q|=2r-a.
\]
Consequently,
\[
 \sum_{p\in Q}\delta_p
 =\frac{1}{P_{\mathcal B'}(x)}
   \sum_{E\in\G_{t+r}^{\partial}}
   |C_E\cap Q|P_{\C_s(E)}(x)
 \ge r-a+1,
\]
and therefore
$\max_p\delta_p\ge(r-a+1)/(2r-a)\ge1/r$.

Now suppose there is no off-middle boundary rank.  Since
$\rho(\F)>\rho_r$, not all nonboundary generators can have size at
least $t+r$; otherwise every edge of $\F$ would contain at least $t+r$ points
of $[s]$, and $\F$ would be a subfamily of $\A_r$.  Choose
$T\in\G^{\circ}$ with $|T|=t+r-m$, $m\ge1$, and let
$Q=[s-1]\setminus T$.  Since every middle-rank generator $E$ $t$-intersects
$T$,
\[
 |C_E\cap Q|\ge m,
 \qquad |Q|=r+m-1.
\]
Consequently,
\[
 \sum_{p\in Q}\delta_p
 =\frac{1}{P_{\mathcal B'}(x)}
   \sum_{E\in\G_{t+r}^{\partial}}
   |C_E\cap Q|P_{\C_s(E)}(x)
 \ge m,
\]
and hence $\max_p\delta_p\ge m/(r+m-1)\ge1/r$.
\end{proof}

For $t\ge2$, the universal gain estimate converts
\eqref{eq:delta-bound} into the required comparison of polynomial values.

\begin{corollary}\label{cor:profitable-middle}
Let $t\ge2$.  Under the assumptions of
Lemma~\ref{lem:frequent-omission}, there exists $a_0\in[s-1]$ such that
\begin{equation}\label{eq:middle-ledger}
 P_{\mathcal N_{a_0}'}(y)
 =\Gamma_{k-t-r}(\eta)
  P_{\mathcal B_{a_0}'}(x)
 \ge P_{\mathcal B'}(x).
\end{equation}
The final inequality in~\eqref{eq:middle-ledger} is strict.  Consequently, if the
middle rank is the entire boundary, then
\[
 \rho(\mathcal Q_{a_0}')>\rho(\F).
\]
\end{corollary}

\begin{proof}
Choose $a_0$ as in Lemma~\ref{lem:frequent-omission}.  Then
\[
 \frac1{\delta_{a_0}}
 \le\min\left\{\frac{s-1}{r},r\right\}=L_{t,r}.
\]
Dominance over the Frankl families gives~\eqref{eq:ground-set-gap} and
\eqref{eq:basic-numerical-condition}; hence
Lemma~\ref{lem:middle-envelope} yields
$\Gamma_{k-t-r}(\eta)\ge g_{k-t-r}>L_{t,r}$.  By
\eqref{eq:delta-definition},
\[
 P_{\mathcal B_{a_0}'}(x)
 =\delta_{a_0}P_{\mathcal B'}(x).
\]
Together with~\eqref{eq:old-new-trace-polynomials}, this proves the strict
form of~\eqref{eq:middle-ledger}.

If this is the entire boundary, then
$\mathcal N_{a_0}'\cap\F^{\circ}=\emptyset$ by the antichain
property.  Lemma~\ref{lem:tail-symmetrization},
\eqref{eq:trace-decomposition}, and~\eqref{eq:middle-ledger} give
\[
 \begin{aligned}
 P_{\mathcal Q_{a_0}'}(y)
 &\ge P_{\F^{\circ}}(y)+P_{\mathcal N_{a_0}'}(y)\\
 &>P_{\F^{\circ}}(x)+P_{\mathcal B'}(x)
  =P_\F(x)=\rho(\F),
 \end{aligned}
\]
which proves the last assertion.
\end{proof}

The off-diagonal selection chooses an entire rank, whereas the diagonal
selection chooses the generators sharing a common omission.  These are the
two components of a single pull; we combine them next.

\subsection{The simultaneous boundary pull}\label{subsec:simultaneous-pull}

Assume that the full $t$-star dominates the Frankl families and that
\eqref{eq:candidate-separation} holds whenever a middle rank is present.  For
each unequal complementary pair of boundary ranks, use
Lemma~\ref{lem:profitable-off-middle} to select a rank satisfying
\eqref{eq:off-middle-polynomial}.  If a middle rank is present and
$P_{\mathcal B'}(x)>0$, use
Corollary~\ref{cor:profitable-middle} to choose a common omitted coordinate and
select precisely the middle-rank generators omitting it; the $t=1$
counterpart is given in Section~\ref{sec:t-one}.  If
$P_{\mathcal B'}(x)=0$, discard the middle rank.  Let
$\mathcal S\subseteq\G^{\partial}$ be the collection of all selected boundary
generators, and write
\[
 \mathcal H:=\mathcal H(\mathcal S),
 \qquad
 \mathcal Q:=\mathcal Q(\mathcal S)
\]
as in~\eqref{eq:macro-generators}.  We call this the
\emph{simultaneous boundary pull}: all boundary-rank orbits are treated at
once, and its output is visibly generated inside $[s-1]$.

The first property required of this operation is that the selected shadows
remain a valid $t$-intersecting generating family.

\begin{lemma}\label{lem:macro-intersection}
If $s>t$, then the generating family $\mathcal H(\mathcal S)$ is
$t$-intersecting.  Consequently, $\mathcal Q(\mathcal S)$ is a
$t$-intersecting $k$-uniform family generated inside $[s-1]$.
\end{lemma}

\begin{proof}
No boundary generator has size exactly $t$ when $s>t$.  Indeed, a
$t$-element generator would be contained in every generator by
Lemma~\ref{lem:standard-generator}(3); the antichain would then consist only
of this set.  Since $\F$ is left-compressed, that set would be $[t]$, forcing
$s=t$.

An old generator $G\in\G^{\circ}$ and a new generator
$A^-=A\setminus\{s\}$ satisfy
\[
 |G\cap A^-|=|G\cap A|\ge t,
\]
because $s\notin G$.  Consider two new generators $A^-,B^-$.  If $A,B$ lie
in selected off-middle ranks and $|A\cap B|=t$, then
Lemma~\ref{lem:tight-pairs} forces their ranks to be complementary.  The
construction selects at most one rank from each unequal complementary pair,
so this is impossible.  Hence $|A\cap B|\ge t+1$, and deleting their common
point $s$ leaves $|A^-\cap B^-|\ge t$.  The same argument applies to one
off-middle generator and one middle-rank generator, because the complement
of the middle rank is the middle rank itself.

Finally, two selected middle-rank generators both omit $p_0$.  Their shadows
have size $t+r-1$ inside the $(t+2r-2)$-element set
$[s-1]\setminus\{p_0\}$, and therefore intersect in at least
\[
 2(t+r-1)-(t+2r-2)=t
\]
points.
\end{proof}

The second required property is that this reduction of the support does not
decrease the spectral polynomial at the constructed vector.

\begin{lemma}\label{lem:spectral-ledger}
With $x$ and the symmetrized vector $y$ from
\eqref{eq:tail-equalization},
\begin{equation}\label{eq:simultaneous-pull-objective}
 P_{\mathcal Q}(y)\ge P_{\F}(x)=\rho(\F),
 \qquad
 \rho(\mathcal Q)\ge\rho(\F).
\end{equation}
If $s>t$, both inequalities in
\eqref{eq:simultaneous-pull-objective} are strict.  Hence, for $s>t$,
$\mathcal Q(\mathcal S)$ is a $t$-intersecting family generated inside
$[s-1]$ with spectral radius strictly larger than that of $\F$.
\end{lemma}

\begin{proof}
For every $A\in\G^{\partial}$, all vertices in $A$ are nonisolated and hence
have positive coordinates by Proposition~\ref{prop:perron-shift}(4).  If
$|A|=k-q$ with $q>0$, then $q<n-s$ by~\eqref{eq:q-range}, and every tail
vertex belongs to a member of $\C_s(A)$; thus $\beta>0$.  If $q=0$, the
factor $\beta^q$ in~\eqref{eq:old-boundary-contribution} is absent.  Hence
\begin{equation}\label{eq:positive-boundary-contribution}
 P_{\C_s(A)}(x)>0
 \qquad(A\in\G^{\partial}).
\end{equation}

By~\eqref{eq:trace-decomposition} and the definition of $\mathcal B_q$,
\begin{equation}\label{eq:old-ledger}
 P_{\F}(x)=P_{\F^{\circ}}(x)+\sum_qP_{\mathcal B_q}(x).
\end{equation}
For distinct selected generators $A$, the families
$\C_{s-1}(A\setminus\{s\})$ are pairwise disjoint.  Moreover,
\begin{equation}\label{eq:new-traces-disjoint}
 \mathcal N(\mathcal S)\cap\F^{\circ}=\emptyset,
\end{equation}
because a member of the intersection would imply that a nonboundary minimal
generator is properly contained in a boundary generator.

For every unequal complementary pair, inequality
\eqref{eq:off-middle-polynomial} compares the selected part of
$P_{\mathcal N(\mathcal S)}(y)$ with the two corresponding terms in
\eqref{eq:old-ledger}.  For the middle rank, the corresponding comparison is
\eqref{eq:middle-ledger} when $t\ge2$ and
\eqref{eq:t-one-middle-ledger} when $t=1$.  Summing these inequalities over
the boundary-rank orbits gives
\begin{equation}\label{eq:new-versus-old-boundary}
 P_{\mathcal N(\mathcal S)}(y)
 \ge\sum_qP_{\mathcal B_q}(x).
\end{equation}
Using~\eqref{eq:new-traces-disjoint}, Lemma~\ref{lem:tail-symmetrization},
\eqref{eq:new-versus-old-boundary}, and~\eqref{eq:old-ledger}, we obtain
\[
 \begin{aligned}
 P_{\mathcal Q}(y)
 &\ge P_{\F^{\circ}}(y)+P_{\mathcal N(\mathcal S)}(y)\\
 &\ge P_{\F^{\circ}}(x)+\sum_qP_{\mathcal B_q}(x)
  =P_{\F}(x)=\rho(\F).
 \end{aligned}
\]
Since $\sum_i y_i^k=1$, we also have
$\rho(\mathcal Q)\ge P_{\mathcal Q}(y)$, as required.

If $s>t$, the boundary $\G^{\partial}$ is nonempty.  Every one of
its nonempty layers has positive polynomial value by
\eqref{eq:positive-boundary-contribution}.  For each unequal
complementary pair, inequality~\eqref{eq:off-middle-polynomial} is strict by
Lemma~\ref{lem:profitable-off-middle}; the middle-rank inequality is strict
by Corollary~\ref{cor:profitable-middle} when $t\ge2$ and by
Corollary~\ref{cor:t-one-profitable-middle} when $t=1$.  At least one such
rank orbit is present, so $P_{\mathcal Q}(y)>P_{\F}(x)$ and hence
$\rho(\mathcal Q)>\rho(\F)$.
\end{proof}

\subsection{The Frankl-family comparison step}

The support-reduction argument is most transparent when isolated in the
following intermediate proposition.  Its hypothesis still asks the star to
dominate all the Frankl families; Proposition
\ref{prop:propagation} will immediately reduce that hypothesis to the single
comparison $\rho_0\ge\rho_1$.  The strict form records the equality case.

\begin{proposition}\label{prop:candidate-comparison}
Let $2\le t\le k$ and $n>2k-t$.  If the full $t$-star is
dominant among the Frankl families, then it is spectrally extremal among all $t$-intersecting
subfamilies of $\binom{[n]}k$.  If, moreover,
\begin{equation}\label{eq:strict-candidate-dominance}
 \rho_0>\rho_r\qquad(1\le r\le k-t),
\end{equation}
then every spectral extremal structure is a full $t$-star, up to permutation.
Conversely, a spectrally extremal full $t$-star must dominate the Frankl
families.
\end{proposition}

\begin{proof}
The case $t=k$ is immediate, so assume $t<k$.  Let
\[
 \rho^*:=\max\left\{\rho(\mathcal B):
 \mathcal B\subseteq\binom{[n]}k\text{ is $t$-intersecting}\right\}.
\]
Suppose for contradiction that $\rho^*>\rho_0$.  Choose an
inclusion-maximal left-compressed spectral extremal structure $\F$ whose minimal
generating antichain has the smallest possible support length $s$, as in
Proposition~\ref{prop:perron-shift}.  Lemma~\ref{lem:tight-pairs} gives
$s\le2k-t$.

If $s=t$, then every generator has size at least $t$ and lies in the
$t$-set $[s]$.  Hence $\G=\{[t]\}$ and $\F=\A_0$, contradicting
$\rho(\F)=\rho^*>\rho_0$.  Thus $s>t$.

Dominance over the Frankl families gives~\eqref{eq:ground-set-gap} and
\eqref{eq:basic-numerical-condition} by
Lemma~\ref{lem:numerical-consequences}.  Hence the off-diagonal selection of
Lemma~\ref{lem:profitable-off-middle} and the diagonal selection of
Corollary~\ref{cor:profitable-middle} are both available; in the latter case
\eqref{eq:candidate-separation} follows from
$\rho(\F)=\rho^*>\rho_0\ge\rho_r$.  Construct
$\mathcal Q$ by the simultaneous boundary pull.  By
Lemma~\ref{lem:macro-intersection}, it is $t$-intersecting and is generated
inside $[s-1]$.  Since $s>t$, Lemma~\ref{lem:spectral-ledger} gives
\[
 \rho(\mathcal Q)>\rho(\F)=\rho^*,
\]
contradicting the definition of $\rho^*$.

Therefore $\rho^*\le\rho_0$, and dominance over the Frankl families gives
equality.

Now assume~\eqref{eq:strict-candidate-dominance} and suppose that a non-star
spectral extremal structure exists.  Dominance over the Frankl families gives
$n\ge2k-t+2$.  Use the second alternative of
Proposition~\ref{prop:perron-shift} to choose an inclusion-maximal,
left-compressed non-star extremal structure $\F$.  Its support length satisfies $s>t$,
because $s=t$ would force $\G=\{[t]\}$ and $\F=\A_0$.  Moreover,
\[
 \rho(\F)=\rho_0>\rho_r\qquad(1\le r\le k-t),
\]
so~\eqref{eq:candidate-separation} holds for any middle rank.  The simultaneous
boundary pull is therefore available.  Lemmas~\ref{lem:macro-intersection}
and~\ref{lem:spectral-ledger} produce a $t$-intersecting family $\mathcal Q$
with
\[
 \rho(\mathcal Q)>\rho(\F)=\rho_0,
\]
contradicting extremality.  Hence every extremal structure is a full $t$-star.

Conversely, a spectrally extremal full $t$-star must dominate every $\A_r$,
because each $\A_r$ is itself $t$-intersecting.
\end{proof}

The intermediate proposition and the numerical propagation now give the main
result in exactly the form stated in the introduction.

\begin{proof}[Proof of Theorem~\ref{thm:first-candidate}]
Assume first that $\rho_0\ge\rho_1$.  Proposition
\ref{prop:propagation} gives
\[
 \rho_0\ge\rho_1\ge\cdots\ge\rho_{k-t},
\]
so the full $t$-star dominates the Frankl families.  Proposition
\ref{prop:candidate-comparison} then makes it spectrally extremal.

If $\rho_0>\rho_1$, the same propagation gives
$\rho_0>\rho_r$ for every $1\le r\le k-t$.  The strict part of
Proposition~\ref{prop:candidate-comparison} therefore shows that every
extremal structure is a full $t$-star.  If $\rho_0=\rho_1$, then $\A_1$ has the same
radius as the now extremal family $\A_0$, so both are extremal structures and
uniqueness fails.

Conversely, if the full $t$-star is spectrally extremal, it must in particular
dominate the admissible $t$-intersecting family $\A_1$, and hence
$\rho_0\ge\rho_1$.  The same argument shows directly why the condition is
sharp: when $\rho_0<\rho_1$, the family $\A_1$ is a strict counterexample.
\end{proof}

\begin{proof}[Proof of Corollary~\ref{cor:explicit-threshold}]
The right-hand side of~\eqref{eq:explicit-threshold} exceeds $2k-t$, so the
standing range of Theorem~\ref{thm:first-candidate} and the pull lemmas is
satisfied.  Assumption~\eqref{eq:explicit-threshold} gives
\begin{equation}\label{eq:explicit-tail-bounds}
 \begin{aligned}
 n-t-1&\ge (t+1)(k-t)
 +\left\lceil(t+1)\log(t+1)\right\rceil,\\
 n-t-2&\ge (t+1)(k-t)
 +\left\lceil(t+1)\log(t+1)\right\rceil-1.
 \end{aligned}
\end{equation}

Suppose first that $k-t\ge2$.  For the pull from $\A_1$ to $\A_0$,
Lemma~\ref{lem:universal-gain} gives
\begin{equation}\label{eq:first-pull-gain}
 g_{k-t-1}^k
 =\left(\frac{n-t-1}{k-t}\right)^t
  \left(\frac{n-t-2}{k-t-1}\right)^{k-t-1}.
\end{equation}
Moreover,
\[
 \frac{n-t-2}{k-t-1}-\frac{n-t-1}{k-t}
 =\frac{n-k-1}{(k-t)(k-t-1)}>0
\]
by~\eqref{eq:explicit-tail-bounds}.  Therefore, using
\[
 \left\lceil(t+1)\log(t+1)\right\rceil
 \ge(t+1)\log(t+1),
\]
\begin{equation}\label{eq:closed-gain-lower}
 g_{k-t-1}^k
 > \left(\frac{n-t-1}{k-t}\right)^{k-1}
 \ge (t+1)^{k-1}
      \left(1+\frac{\log(t+1)}{k-t}\right)^{k-1}.
\end{equation}
We verify that the last factor is at least $t+1$.  If $t\ge3$, then
$t-1\ge\log(t+1)$, and $\log(1+x)\ge x/(1+x)$ gives
\[
 (k-1)\log\left(1+\frac{\log(t+1)}{k-t}\right)
 \ge
 \frac{(k-1)\log(t+1)}{k-t+\log(t+1)}
 \ge\log(t+1).
\]
If $t=2$, then
$\log3<4/3\le2(k-t)/(k-t+1)$ for $k-t\ge2$.  Using
$\log(1+x)\ge x-x^2/2$, we obtain
\[
 (k-t+1)\log\left(1+\frac{\log3}{k-t}\right)
 \ge \log3+
 \frac{\log3}{2(k-t)^2}
 \bigl(2(k-t)-(k-t+1)\log3\bigr)
 \ge\log3.
\]
Thus~\eqref{eq:closed-gain-lower} yields $g_{k-t-1}>t+1$.

It remains to treat $k-t=1$.  Here $k=t+1$, and
Lemma~\ref{lem:universal-gain} gives
\[
 g_0^k=(n-t-1)^t
 \ge (t+1)^t(1+\log(t+1))^t.
\]
For $t\ge3$, the inequalities
$\log(1+\log(t+1))>\log(t+1)/(1+\log(t+1))$ and
$t>1+\log(t+1)$ imply
\[
 t\log(1+\log(t+1))>\log(t+1).
\]
For $t=2$, the same conclusion follows directly from
$(1+\log3)^2>3$.  Hence $g_0^k>(t+1)^{t+1}$, so again
$g_0>t+1$.

In both cases the strict form of the hypothesis of
Lemma~\ref{lem:adjacent-pull} holds for
$r=1$, because $B_{t,1}=t+1$.  Pulling $\A_1$ to $\A_0$ therefore gives
$\rho_0>\rho_1$: the boundary trace classes of $\A_1$ have positive values
under $P$ at its Perron vector, and the corresponding polynomial inequality
is strict.  Theorem~\ref{thm:first-candidate} now proves the
corollary.
\end{proof}

\section{The case \texorpdfstring{$t=1$}{t=1}}\label{sec:t-one}

For ordinary intersection, the first comparison does not control the whole
Frankl-family sequence at the endpoint $n=2k$: the star may beat $\A_1$ while
losing to a later Frankl family.  For example, formula
\eqref{eq:candidate-orbit-formula} gives, at $(n,k)=(8,4)$,
$\rho_0\approx18.539>\rho_1\approx18.486$, whereas $\rho_3=20$.  We first
record the special gain estimate that completes the diagonal pull for $t=1$.
Once $n\ge2k+1$, the same estimate also makes every adjacent Frankl-family
comparison point in the same direction.

\begin{lemma}
\label{lem:t-one-gain}
Let $t=1$ and $n\ge2k+1$.  For $1\le r\le k-1$, let
$L_{1,r}:=\min\{r,2\}$.  Then the corresponding universal gain satisfies
$g_{k-r-1}>L_{1,r}$.  In particular, $g_{k-r-1}>2$ whenever $r\ge2$.
\end{lemma}

\begin{proof}
For $r=1$ the desired strict bound is $1$.  If $k-r-1\ge1$, both factors in
\eqref{eq:universal-gain} exceed $1$; if $k-r-1=0$, the same conclusion
follows from~\eqref{eq:gain-q-zero}.

Now let $r\ge2$, so $L_{1,r}=2$.  If $k-r-1=0$, then $k=r+1$, and
 \[
 g_0=(n-2r)^{r/(r+1)}\ge3^{r/(r+1)}>2,
 \]
where the last inequality follows from $3^r>2^{r+1}$ for $r\ge2$.
If $k-r-1\ge1$, formula~\eqref{eq:universal-gain}, together with
$n-1-2r\ge2(k-r-1)+2$, gives
\[
 \begin{aligned}
 g_{k-r-1}^k
 &\ge
 \left(2+\frac1{k-r}\right)^r
 \left(2+\frac2{k-r-1}\right)^{k-r-1}\\
 &\ge2^{k-1}\left(1+\frac1{2(k-r)}\right)
              \left(1+\frac1{k-r-1}\right)^{k-r-1}\\
 &>2^k.
 \end{aligned}
\]
Here the last inequality uses
$(1+1/(k-r-1))^{k-r-1}\ge2$ and
$1+1/(2(k-r))>1$.
This proves the assertion.
\end{proof}

For $t=1$, this special gain estimate supplies exactly the same combinatorial
diagonal pull as in the preceding section.

\begin{corollary}
\label{cor:t-one-profitable-middle}
Let $t=1$.  Under the assumptions of
Lemma~\ref{lem:frequent-omission}, there exists $p_0\in[s-1]$ such that
\begin{equation}\label{eq:t-one-middle-ledger}
 P_{\mathcal N_{p_0}'}(y)
 =\Gamma_{k-r-1}(\eta)
  P_{\mathcal B_{p_0}'}(x)
 \ge P_{\mathcal B'}(x).
\end{equation}
The final inequality in~\eqref{eq:t-one-middle-ledger} is strict.
\end{corollary}

\begin{proof}
Dominance over the Frankl families gives $n\ge2k+1$ by
Lemma~\ref{lem:numerical-consequences}.  Choose $p_0$ attaining the maximum
in~\eqref{eq:delta-bound}.  Then
\[
 \frac1{\delta_{p_0}}
 \le\min\left\{\frac{s-1}{r},r\right\}=L_{1,r}.
\]
Lemma~\ref{lem:t-one-gain} yields
$\Gamma_{k-r-1}(\eta)\ge g_{k-r-1}>L_{1,r}$.  Equations
\eqref{eq:delta-definition} and~\eqref{eq:old-new-trace-polynomials} give
\[
 P_{\mathcal N_{p_0}'}(y)
 =\Gamma_{k-r-1}(\eta)\delta_{p_0}
  P_{\mathcal B'}(x)
 >P_{\mathcal B'}(x),
\]
as required.
\end{proof}

The two boundary selections give the $t=1$ version of the intermediate
Frankl-family comparison step.  It is recorded here only for use in the proof of
Theorem~\ref{thm:t-one}.

\begin{proposition}\label{prop:t-one-candidate-comparison}
Let $k\ge2$ and $n\ge2k$.  If the full star dominates the Frankl families,
then it is spectrally extremal among all intersecting subfamilies of
$\binom{[n]}k$.
If, moreover, $n\ge2k+1$ and
\begin{equation}\label{eq:t-one-strict-candidate-dominance}
 \rho_0>\rho_r\qquad(1\le r\le k-1),
\end{equation}
then every spectral extremal structure is a full star, up to permutation.
\end{proposition}

\begin{proof}
Suppose, as in the proof of Proposition~\ref{prop:candidate-comparison}, that a
spectral extremal structure $\F$ satisfies
$\rho(\F)>\rho_0$, and choose it with inclusion-maximality, left-compression,
and minimal support length $s$.  Lemma~\ref{lem:tight-pairs} gives
$s\le2k-1$.  If $s=1$, then $\G=\{\{1\}\}$ and $\F=\A_0$, contrary to
$\rho(\F)>\rho_0$; hence $s>1$.  Dominance over the Frankl families gives
$n\ge2k+1$ by
Lemma~\ref{lem:numerical-consequences}.

Lemma~\ref{lem:profitable-off-middle} handles every unequal complementary
pair, while Corollary~\ref{cor:t-one-profitable-middle} handles the possible
diagonal rank.  Use these selections in the simultaneous boundary pull of
Subsection~\ref{subsec:simultaneous-pull}.  Lemma~\ref{lem:macro-intersection}
gives a $1$-intersecting family $\mathcal Q$ generated inside $[s-1]$, and
Lemma~\ref{lem:spectral-ledger} gives
\[
 \rho(\mathcal Q)>\rho(\F),
\]
contradicting the maximality of $\F$.  Hence the full star is spectrally
extremal.

Now assume~\eqref{eq:t-one-strict-candidate-dominance} and suppose that a
non-star spectral extremal structure exists.  Since $n\ge2k+1$, the second alternative
of Proposition~\ref{prop:perron-shift} gives an inclusion-maximal,
left-compressed non-star extremal structure $\F$ with minimal support length $s$.  We
have $s>1$, since $s=1$ would force $\G=\{\{1\}\}$ and $\F=\A_0$.  Moreover,
\[
 \rho(\F)=\rho_0>\rho_r\qquad(1\le r\le k-1),
\]
so the separation hypothesis holds for every possible middle rank.
The simultaneous boundary pull is therefore available.  Its boundary is
nonempty, and Lemmas~\ref{lem:macro-intersection} and
\ref{lem:spectral-ledger} produce an intersecting family $\mathcal Q$ with
\[
 \rho(\mathcal Q)>\rho(\F)=\rho_0,
\]
contradicting extremality.  Thus every extremal structure is a full star.
\end{proof}

The same estimate also determines the order of all Frankl families once
$n\ge2k+1$.

\begin{proposition}
\label{prop:t-one-sequence}
Let $k\ge2$ and $n\ge2k+1$.  Then the full star dominates every Frankl
family.  Unless $(n,k)=(5,2)$, their radii are
strictly decreasing:
\begin{equation}\label{eq:t-one-strict-sequence}
 \rho_0>\rho_1>\cdots>\rho_{k-1}.
\end{equation}
In the exceptional case, $\rho_0=\rho_1=2$.
\end{proposition}

\begin{proof}
For every $2\le r\le k-1$, $B_{1,r}=2$, and
Lemma~\ref{lem:t-one-gain} gives $g_{k-r-1}>2$.  Hence the strict part of
Lemma~\ref{lem:adjacent-pull} gives $\rho_{r-1}>\rho_r$.

It remains to compare $\rho_0$ and $\rho_1$.  If $k\ge3$, then
$n-3\ge2(k-2)+2$, so formula~\eqref{eq:universal-gain} gives
\[
 g_{k-2}^k\ge
 \left(2+\frac1{k-1}\right)
 \left(2+\frac2{k-2}\right)^{k-2}
 >2^{k-1}\left(1+\frac1{k-2}\right)^{k-2}
 \ge2^k.
\]
Thus $g_{k-2}>2=B_{1,1}$ and the strict part of
Lemma~\ref{lem:adjacent-pull} gives $\rho_0>\rho_1$.  Together with the strict
comparisons for $r\ge2$ already proved, this
yields~\eqref{eq:t-one-strict-sequence}.  If
$k=2$, the only Frankl families are the star $K_{1,n-1}$ and the triangle $K_3$,
whose spectral radii are $\sqrt{n-1}$ and $2$, respectively.  Hence
$\rho_0>\rho_1$ for $n\ge6$, while at $(n,k)=(5,2)$ one has
$\rho_0=\rho_1=2$.
\end{proof}

\begin{proof}[Proof of Theorem~\ref{thm:t-one}]
Suppose first that $n\ge2k+1$.  Proposition~\ref{prop:t-one-sequence} gives
dominance over the Frankl families, and
Proposition~\ref{prop:t-one-candidate-comparison} gives
$\rho(\F)\le\rho_0$ for every intersecting $\F$.

If $(n,k)\ne(5,2)$, Proposition~\ref{prop:t-one-sequence} also gives
\eqref{eq:t-one-strict-candidate-dominance}, so the strict part of
Proposition~\ref{prop:t-one-candidate-comparison} shows that equality holds
only for a full star, up to permutation.  At $(n,k)=(5,2)$, the full star
$K_{1,4}$ and the triangle $K_3$ both have spectral radius $2=\rho_0$.
To see that there are no further equality cases, regard an intersecting
$2$-uniform family as a graph.  A graph with at most one edge is a subgraph
of a star.  Otherwise choose two edges $ab$ and $ac$.  Every other edge
either contains $a$ or is the edge $bc$.  If $bc$ occurs, every edge
containing $a$ must also meet $bc$ and hence is $ab$ or $ac$, so the graph is
a subgraph of the triangle $abc$.  If $bc$ does not occur, all edges have the
common endpoint $a$, so the graph is a subgraph of a star.  A proper
subgraph of $K_{1,4}$ has spectral radius at most $\sqrt3<2$, and a proper
subgraph of $K_3$ has spectral radius at most $\sqrt2<2$.  Hence the full
star and the triangle, together with the requisite isolated vertices, are
exactly the equality structures.

It remains to consider $n=2k$.  Let $m=|\F|$.  The
Erd\H{o}s--Ko--Rado theorem gives
\[
 m\le \binom{2k-1}{k-1}=\binom{2k-1}{k}.
\]
For $k\ge3$, Bai and Lu~\cite{BaiLu}, using the same adjacency-tensor
normalization as~\eqref{eq:spectral-polynomial}, proved that a $k$-uniform
hypergraph with $m$ edges has spectral radius at most $f_k(m)$, where $f_k$
is strictly increasing and
\[
 f_k\left(\binom{a}{k}\right)=\binom{a-1}{k-1}
 \qquad\text{for every integer }a\ge k.
\]
Their equality theorem says that equality at a binomial edge count occurs
only for the corresponding complete $k$-graph together with isolated
vertices.  Consequently,
\[
 \rho(\F)\le f_k(m)
 \le f_k\left(\binom{2k-1}{k}\right)
 =\binom{2k-2}{k-1},
\]
and equality holds if and only if, up to permutation,
$\F=\binom{[2k-1]}k$.  When $k=2$,
\eqref{eq:spectral-polynomial} is the usual adjacency-matrix Rayleigh
quotient, and Stanley's inequality~\cite{Stanley},
\[
 \rho(\F)\le\frac{\sqrt{1+8m}-1}{2},
\]
and its equality characterization give the same conclusion: $m\le3$,
$\rho(\F)\le2$, and equality holds only for a triangle together with an
isolated vertex.  Hence in every case the last Frankl family is the unique
endpoint extremal structure up to permutation, with
\[
 \rho_{k-1}=\binom{2k-2}{k-1}.
\]

It remains only to verify the strict comparison with the star asserted in the
statement.  Formula~\eqref{eq:endpoint-radii} gives
\[
 \left(\frac{\rho_{k-1}}{\rho_0}\right)^k
 =\frac{k^k}{(2k-1)(k-1)^{k-1}}>1,
\]
because
\[
 k\left(1+\frac1{k-1}\right)^{k-1}
 \ge2k>2k-1.
\]
Thus $\rho_{k-1}>\rho_0$, as required.
\end{proof}

\section{The spectral complete-intersection conjecture}\label{sec:conjecture}

Theorem~\ref{thm:first-candidate} settles the extremal value whenever
$t\ge2$ and $\rho_0\ge\rho_1$ and proves uniqueness up to permutation
throughout the strict range $\rho_0>\rho_1$, while
Theorem~\ref{thm:t-one} proves the full
spectral complete-intersection value statement for ordinary intersection: a
maximizing Frankl family is $\A_0$ for $n\ge2k+1$, with the full star unique
up to permutation except for the additional triangle at $(n,k)=(5,2)$, and
the unique extremal structure up to permutation is $\A_{k-1}$ for $n=2k$.
For general $t$, the natural next statement is that the
list of Frankl families always contains a spectral extremal structure.

\begin{conjecture}\label{conj:SCIT}
Let $1\le t\le k$ and $n>2k-t$.  Among all $t$-intersecting subfamilies of
$\binom{[n]}k$, at least one spectral extremal structure is a Frankl family
$\A_r$.  Equivalently, every $t$-intersecting family
$\F\subseteq\binom{[n]}k$ satisfies
\begin{equation}\label{eq:SCIT}
 \rho(\F)\le
 \max_{0\le r\le k-t}
 \rho\left(
 \left\{F\in\binom{[n]}k:|F\cap[t+2r]|\ge t+r\right\}
 \right).
\end{equation}
\end{conjecture}

The value statement~\eqref{eq:SCIT} should be separated from its equality
theory.  A natural strong form predicts that if a unique index $r$ maximizes
the right-hand side, then every spectral extremal structure equals $\A_r$ up to
permutation.
Theorem~\ref{thm:first-candidate} confirms this strong form when the unique
maximizing Frankl family is $\A_0$.  When several Frankl-family radii
coincide, a separate
equality analysis may still be required.

Theorem~\ref{thm:first-candidate} resolves the value part of
Conjecture~\ref{conj:SCIT} on the entire side detected by
$\rho_0\ge\rho_1$ when $t\ge2$, including rigidity in the strict range, while
Theorem~\ref{thm:t-one} proves the conjecture in full for $t=1$, including the
complete equality classification.  A
targeted version of the same Perron-aware push--pull mechanism is the
remaining step toward the full spectral complete intersection theorem for
$t\ge2$ outside the star-dominant range.

\section*{Acknowledgements}

The authors gratefully acknowledge the Fourth ECOPRO Student Research
Program, held at the Institute for Basic Science (IBS) in summer 2026,
for its support. The authors acknowledge the use of AI tools during the
exploratory stage of this project. All mathematical arguments and proofs
presented in the final manuscript were developed and rigorously verified
by the authors. The authors take full responsibility for the content of
the manuscript.

\end{document}